\documentclass[12pt,leqno]{amsart} 

\usepackage{amsmath}
\usepackage{amssymb}
\usepackage{amsthm} 
\usepackage[pagebackref]{hyperref} 
\usepackage[T1]{fontenc}
\hypersetup{
	colorlinks=true, 
	linkcolor=blue,  
	citecolor=green, 
}
\usepackage{esint} 
\usepackage{enumerate} 
\usepackage[shortlabels]{enumitem}
\usepackage{mathtools}
\usepackage{comment} 
\usepackage[english]{babel}
\usepackage[dvipsnames]{xcolor}

\numberwithin{equation}{section}

\newtheorem{lemma}{Lemma}[section]

\newtheorem{thm}[lemma]{Theorem}
\newtheorem{cor}[lemma]{Corollary}
\theoremstyle{definition}
\newtheorem{rmk}[lemma]{\bf Remark}
\newtheorem{defn}[lemma]{\bf Definition}

\newcommand{\N}{\mathbb{N}}
\newcommand{\Z}{\mathbb{Z}}

\newcommand{\R}{\mathbb{R}}

\newcommand{\ra}{\rightarrow}
\newcommand{\da}{\downarrow}

\newcommand{\wra}{\rightharpoonup}
\newcommand{\HD}{\mathcal{H}}
\newcommand{\norm}[1]{\left \lVert #1 \right \rVert}
\newcommand{\abs}[1]{\left \lvert #1 \right \rvert}
\renewcommand{\epsilon}{\varepsilon}
\newcommand{\HDE}{\HD^d|_E}
\newcommand{\dist}{\mathrm{dist \, }}
\newcommand{\diam}{\mathrm{diam \,  }}

\newcommand{\family}{\mathcal{F}}

\newcommand{\spt}{\mathrm{spt \, }}
\newcommand{\pushm}[2]{{#1}_\# {#2}}
\newcommand{\cmu}{C_\mu}
\newcommand{\cd}{C_D}
\newcommand{\osc}{\mathcal{O}}
\newcommand{\mubar}{{\tilde{\mu}}}

\DeclareMathOperator{\esssup}{ess \, sup}
\DeclareMathOperator{\essinf}{ess \, inf}
\DeclareMathOperator{\osct}{osc }
\DeclareMathOperator{\gen}{gen }
\DeclareMathOperator{\BMO}{BMO}

\allowdisplaybreaks

\title{Small-Constant Uniform Rectifiability}
\author{Cole Jeznach}
\thanks{C. Jeznach was partially supported by the Simons Collaborations in MPS grant 563916, and NSF DMS grant 2000288. The author would like to thank his advisors Max Engelstein and Svitlana Mayboroda, as well as Guy David for many helpful conversations regarding the main result. The author also thanks the referees for many helpful comments which improved the paper.
}
\address{School of Mathematics, University of Minnesota, Minneapolis, MN, 55455, USA.}
\email{jezna001@umn.edu} 
\date{\today}
\addtocontents{toc}{\protect\setcounter{tocdepth}{1}}

\begin{document}

\begin{abstract}
We provide several equivalent characterizations of locally flat, $d$-Ahlfors regular, uniformly rectifiable sets $E$ in $\R^n$ with density close to $1$ for any dimension $d \in \N$, $1 \le d < n$.
In particular, we show that when $E$ is Reifenberg flat with small constant and has Ahlfors regularity constant close to $1$, then the Tolsa $\alpha$ coefficients associated to $E$ satisfy a small-constant Carleson measure estimate.
This estimate is new, even when $d= n-1$, and gives a new characterization of chord-arc domains with small constant.
\end{abstract}

\maketitle
\tableofcontents

\section{Introduction}
The connection between quantitative properties of elliptic PDEs, harmonic analysis, and geometric measure in the past thirty years has significantly been influenced by the introduction of uniformly rectifiable sets by David and Semmes in the early 90s. At its core, uniform rectifiability of a $d$-dimensional set $E \subset \R^n$ is \textit{the} precise condition on $E$ which guarantees all (sufficiently nice) Calder\'{o}n-Zygmund operators are $L^2$ bounded \cite{DSSIO}. In terms of elliptic boundary value problems, it turns out that uniform rectifiability of the boundary $\partial \Omega$ of a domain $\Omega$ arises naturally as one of the sharp geometric conditions under which one can solve the Laplace-Dirichlet problem on $\Omega$ with singular (i.e., $L^p$) boundary data (see \cite{AHMMT20} for a recent result, but also the series of works \cite{HM14, HMUT14, HMMM14}). At their core, though, uniformly rectifiable sets have many equivalent geometric characterizations, all of which quantify (in some sense) the $d$-rectifiability of $E$ at different points $x \in E$ and scales $r >0$. 

Just to list two such examples, a $d$-Ahlfors regular set $E \subset \R^n$ is $d$-uniformly rectifiable if and only if the Tolsa $\alpha_{\HD^d|_E}$ numbers satisfy the Carleson measure estimate \cite{TOLSAALPHA}
\begin{align}\label{eqn:tolsa_alpha}
\sup_{x \in E, \, r > 0} r^{-d} \int_0^r \int_{E \cap B(x,r)} \alpha_{\HD^d|_E}(y,t)^2 \; \dfrac{d\HD^d(y) dt}{t} \le M_1,
\end{align}
for some uniform $M_1 > 0$. Here the $\alpha_{\HD^d|_E}(x,r)$ are bounded coefficients which measure the distance from $E$ to the space of $d$-planes in the ball $B(x, r)$ (see Definition \ref{defn:alpha}), and so the estimate \eqref{eqn:tolsa_alpha} says that for most balls $B(x,r)$ centered on $E$, this distance is quantitatively small in a precise sense.
Of course, in the estimate above, one could take different coefficients (e.g., the so-called $L^1$ beta coefficients, $\beta_1$) and still obtain a characterization \cite{DSSIO}. In terms of a slightly more concrete definition, it turns out that \eqref{eqn:tolsa_alpha} is equivalent to $E$ having ``big pieces of Lipschitz images of subsets of $\R^d$'', which is to say the following: there is some uniform $M_2> 0$ so that for each $x \in E$ and every $r >0$, one can find a Lipschitz mapping $\rho: B_d(0, r) \subset \R^d \ra \R^n$ with Lipschitz norm $\le 1 + M_2$ so that 
\begin{align}\label{eqn:bpbi}
\HD^d(E \cap B(x,r) \cap \rho(B_d(0,r))) \ge (1+M_2)^{-1} \HD^d(B_d(0,r)).
\end{align}
There are many other interesting geometric and analytic characterizations of uniformly rectifiable sets, and we refer the reader to \cite{DSSIO, DSUR} where this is pursued. The goal of the current paper is to take on a systematic study of the quantitative relationship between such constants $M_1$ and $M_2$ in the \textit{small-constant} regime: if $M_1$ is sufficiently small, does it mean that $M_2$ also is? If so, can such a relationship be made quantitative? 

In this paper, we show that this is indeed the case.
In fact, we show that the estimate \eqref{eqn:tolsa_alpha} with small constant $M_1$ (along with good Ahlfors regularity control) characterizes a certain class of Ahlfors regular sets $E \subset \R^n$ of any dimension $1 \le d \le n-1$ that have very good approximations by very flat Lipschitz graphs (Theorem \ref{thm:main}). This approximation property is even stronger than the ``big pieces of Lipschitz images of subsets of $\R^d$'' property mentioned above. We call such sets uniformly rectifiable of dimension $d$ with small constant $\delta >0$ (see Definition \ref{defn:delta_ur}).
Moreover our result is quantitative in that \eqref{eqn:tolsa_alpha} holds with $M_1 = \delta^\theta$ for some dimensional constant $\theta \in (0,1)$ depending only on $n$ and $d$ whenever $E \subset \R^n$ is uniformly rectifiable of dimension $d$ with small constant $\delta$, and a converse holds as well. 
This quantitative Carleson measure estimate serves as an important tool in the upcoming work in \cite{EJLM}, where the authors study the regularity of the Poisson kernel associated to a degenerate elliptic operator outside of Ahlfors regular sets of high co-dimension in $\R^n$.
In addition, our method of proof brings with it several other characterizations. 
In particular, we relate the constant $M_1$ to the control of the oscillation of the tangent planes to $E$ and the Reifenberg flatness of $E$.

These two other characterizations are largely motivated by the of work Semmes \cite{CASI, CASII} (and later Blatt \cite{BLATT}) on chord-arc surfaces with small constant as well as Kenig and Toro \cite{KT97, KT99} in their study of the Poisson kernel regularity for chord-arc domains with small (and vanishing) constant.
In particular, the work of Kenig and Toro showed that chord-arc domains with small constant in many ways serve as an appropriate substitute for $C^1$ domains in the study of boundary value problems for elliptic PDE below the Lipschitz thresh-hold.
It turns out that under a global assumption of Reifenberg flatness of a domain $\Omega$, the Poisson kernel $k$ associated to the Laplace operator on $\Omega$ satisfies $\log k \in \mathrm{VMO}(\partial \Omega)$ if and only if the domain $\Omega$ is a chord-arc domain with vanishing constant \cite{KT03}.
This result is the proper analogue (and converse) of the earlier result of Jerison and Kenig, which says that $\log k \in \mathrm{VMO}(\partial \Omega)$ for $C^1$ domains (though, in general $\log k$ need not be continuous or even bounded for such domains) \cite{JK82}.
Since then, chord-arc domains with small constant have continued to be an important geometric object in the study of quantitative properties of solutions to elliptic PDE on rough domains \cite{MT10, HMT10, MPT13, BO17, DLMAinfty, BTZ23}, free boundary problems for elliptic measure \cite{BH16, AMT17, BET20, PT20, BEGTZ22}, and even have corresponding analogues and importance in other PDE settings \cite{LN12, E17, MP21}, and we can only scratch the surface here on the plethora of theory devoted to the study of PDE on such domains.

Since chord-arc domains with small constant $\Omega$ have rich PDE properties, there has been much interest in understanding and providing alternative geometric characterizations of such domains.
Roughly speaking, these are domains whose boundaries locally separate space in two and whose boundaries are Ahlfors regular and bilaterally well-approximated by hyperplanes.
In addition these domains have unit normal with small BMO-norm (see Definition \ref{defn:cad} for a more precise statement) \cite{KT99}.
It is known that such domains also have good Lipschitz graph approximations, and thus their boundaries are closely related to uniformly rectifiable sets of dimension $(n-1)$ with small constant. In fact, when $\Omega$ is a domain satisfying some underlying topological assumptions, we shall use our results to give an alternative characterization of $\Omega$ being a chord-arc domain with small constant using the Carleson measure estimate \eqref{eqn:tolsa_alpha} on $\partial \Omega$ (see Theorem \ref{thm:cad}).

Before rigorously stating the main result, we remark that the relationship between some of the defining characteristics of chord-arc domains with small constant (such as Reifenberg flatness, oscillation of the unit normal, and Lipschitz graph approximations) has been studied and exists in the literature in varying contexts (in the co-dimension one case for chord-arc surfaces and chord-arc domains in \cite{CASI, KT03, HMT10}, and in any co-dimension for smooth embedded hypersurfaces \cite{BLATT}, for example). Still in our main result for uniformly rectifiable sets $E \subset \R^n$ of dimension $d$ and small constant $\delta >0$, we provide proofs that hold for general Ahlfors regular sets of any co-dimension, and we do not impose any topological assumptions on the set $\R^n \setminus E$ apriori. In any case, the characterization in terms of the small constant Carleson measure estimate \eqref{eqn:tolsa_alpha} is new in any dimension and co-dimension. In addition, our techniques provide a systematic way to obtain small-constant Carleson measure estimates such as \eqref{eqn:tolsa_alpha} for coefficients besides the Tolsa $\alpha$ numbers for small constant uniformly rectifiable sets, which we hope to prove useful for small-constant PDE results in the future. Let us now provide enough background to state the main result, Theorem \ref{thm:main}.

\subsection{Main result and outline of the paper} 

In this paper, we always denote the ambient space by $\R^n$, for $n \in \N$, and $d \in \N$ will always be so that $0 < d < n$.
We reserve the notation $A(n,d)$ to denote the collection of all $d$-planes $P \subset \R^n$, and $G(n,d)$ for the Grassmannian of $d$-dimensional subspaces of $\R^n$.
Also, we denote by $\HD^d$ the $d$-dimensional Hausdorff measure on $\R^n$, normalized for notational convenience so that if $P \in A(n,d)$, $x \in P$, and $r >0$, then $\HD^d(B(x,r) \cap P) = r^d$.
Lastly, whenever $P$ is a plane, we denote by $\pi_P: \R^n \ra P$ the orthogonal projection onto the plane $P$.
Let us begin by introducing several related notions of $d$-dimensional sets in $\R^n$ and their geometric regularity that are needed to state our main result.

\begin{defn}\label{defn:ahl_reg}
A Borel measure $\mu$ on $\R^n$ is said to be $d$-Ahlfors regular with constant $\cmu >0$ provided that for each $x \in \spt \mu$ and each $r >0$, one has 
\begin{align} \label{eqn:ahl_reg}
\cmu^{-1} r^d \le \mu (B(x,r)) \le \cmu r^d.
\end{align}
If $E \subset \R^n$ is closed, we say that $E$ is $d$-Ahlfors regular with constant $C_E >0$ if $\HD^d|_E$ is $d$-Ahlfors regular with constant $C_E >0$. Finally, if only the upper (lower) bound holds as above, then we say $\mu$ is upper (lower) $d$-Ahlfors regular with constant $\cmu$. 
\end{defn}

\begin{rmk}\label{rmk:adr_m}
The choice to normalize $\HD^d$ as above, and the role of the constant $\cmu$ in Definition \ref{defn:ahl_reg} is important, since we shall often want to measure how close a $d$-Ahlfors regular measure $\mu$ is to $d$-dimensional surface measure on $\spt \mu$. In particular, we shall often use the phrase ``$d$-Ahlfors regular with small constant'' when the constant $C_\mu > 1$ is very close to $1$, even though the phrase is misleading.
\end{rmk}

Next we introduce Jones' $\beta$ numbers (see \cite{J90, DSSIO}) and Tolsa's $\alpha$ numbers (see \cite{TOLSAALPHA}), which have been studied extensively in relation to rectifiable and uniformly rectifiable measures on $\R^n$ and singular integral operators.
We also introduce the notion of Reifenberg flat sets, which were introduced by Reifenberg in his solution of the Plateu problem \cite{REIF}.

\begin{defn}\label{defn:l1_beta}
If $E$ is a $d$-Ahlfors regular set, then define for $x \in \R^n$ and $r >0$,
\begin{align*}
b\beta_{1, E}(x,r) & : = r^{-d-1} \inf_{P \in A(n,d)}\left( \int_{B(x,r)} \dist(y, P) \; d\HD^d|_E(y) + \int_{B(x,r)} \dist(y,E) \; d \HD^d|_P(y) \right).
\end{align*}
\end{defn}

\begin{defn}\label{defn:wasserstein}
For $\Omega \subset \R^n$ open, denote by $\Lambda(\Omega)$ the space of $1$-Lipschitz functions $f: \R^n \ra \R$ that are compactly supported in $\Omega$. If $\mu$ and $\nu$ are measures on $\R^n$, then we define the localized Wasserstein distance between $\mu$ and $\nu$ in $B(x,r) \subset \R^n$ by 
\begin{align*}
\mathcal{D}_{x,r}(\mu, \nu) \coloneqq \sup_{f \in \Lambda(B(x,r))} \abs{ \int f (d \mu - d\nu) }. 
\end{align*} 
\end{defn}
	
\begin{defn}\label{defn:alpha}
Denote by $\mathrm{Flat}(n,d)$ the set of measures of the form $c \HD^d|_P$ where $c > 0$ and $P \in A(n,d)$. 
If $\mu$ is a $d$-Ahlfors regular measure, then define for $x \in \R^n$ and $r >0$, 
\begin{align*}
\alpha_{\mu}(x,r) & : = r^{-d-1} \inf_{\nu \in \mathrm{Flat}(n,d)}  \mathcal{D}_{x,r}(\mu, \nu).
\end{align*}
In our notation $\alpha_{\mu}$, we omit the dependence on the dimension $d$ of the measure $\mu$, since it shall be clear from context.
\end{defn}

\begin{defn} \label{defn:hd_dist}
We define the normalized local Hausdorff distance for closed sets $E, F \subset \R^n$ that meet $\overline{B(x,r)}$ by 
\begin{align*}
d_{x,r}(E, F) & : = r^{-1} \left( \sup_{y \in E \cap \overline{B(x,r)}} \dist(y, F) + \sup_{y \in F  \cap \overline{B(x,r)}} \dist(y, E) \right).
\end{align*}
With this distance, we define the bilateral beta (infinity) numbers by 
\begin{align*}
b\beta_{\infty, E}(x,r) & : = \inf_P d_{x,r}(E, P),
\end{align*}
where the infimum is taken over all $d$-planes $P \in A(n,d)$ that meet $\overline{B(x,r)}$. Moreover, we say that a closed set $E \subset \R^n$ is $\delta$-Reifenberg flat if $b \beta_{\infty, E}(x,r) \le \delta$ for every $x \in E$ and $r >0$. We warn the reader that this definition of $\delta$-Reifenberg flatness is different than that in \cite{KT99}.
\end{defn}

Finally, we come to the notion of small-constant uniformly rectifiable sets.
\begin{defn} \label{defn:delta_ur}
A closed set $E \subset \R^n$ is $\delta$-uniformly rectifiable of dimension $d$ ($\delta$-UR for short) if $0 < \delta < 1/10$ and the following holds:
\begingroup
\addtolength{\jot}{0.5em}
\begin{align}
& \text{ \parbox[t]{0.8 \textwidth} { for every $x \in E$ and $r >0$, there is a $d$-dimensional Lipschitz graph $\Gamma$ with constant $\le \delta$ so that $\HD^d|_{B(x,r)}(E \Delta \Gamma) \le \delta r^d$ and $\Gamma \cap B(x, r/2) \ne \emptyset$.}}  
\end{align}
\endgroup
Again we usually omit the dimension $d$ since it will be clear from context.
\end{defn}

\begin{rmk} 
In the definition of $\delta$-UR we \textit{impose} that $\delta < 1/10$. This is because if $\delta$ were allowed to be large, the definition would be satisfied for any $d$-Ahlfors regular set $E$, whereas we want the $\delta$-UR condition to be some small-constant quantification uniform rectifiability.
In particular, since $\delta < 1/10$, it is straight-forward to verify that $\delta$-UR sets are $d$-Ahlfors regular with constant close to $1$ (see Lemma \ref{lemma:delta_ur_adr}). Moreover, they satisfy the ``big pieces of Lipschitz graphs condition'' and so $\delta$-UR sets are $d$-uniformly rectifiable as in the sense of David and Semmes \cite{DSUR} with bounded constant (see Definition \ref{defn:ur}).
From the previous discussion it follows that $\delta$-UR sets are $d$-rectifiable, so they have approximate tangent planes for $\HD^d$-almost all $x \in E$ (see Theorem \ref{thm:maggi_blowup}), which we denote by $T(x) \in G(n,d)$.

Notice also that the $\delta$-UR condition is strictly stronger than the ``big pieces of Lipschitz images of subsets of $\R^d$'' condition mentioned in \eqref{eqn:bpbi} with small constant $M_2$. Indeed, if $E = V_1 \cup V_2$ where $V_1$ and $V_2$ are distinct $d$-planes in $\R^n$, then one can check that $E$ is $d$-Ahlfors regular and satisfies \eqref{eqn:bpbi} with $M_2 = 0$ but is not $\delta$-UR of dimension $d$ for $\delta$ small.
\end{rmk}

In the language of the above, our main result is that a set $E \subset \R^n$ is $\delta$-UR of dimension $d$ if and only if one of various other quantities is sufficiently small (with quantitative control).
 We refer the reader to Definition \ref{defn:corona} for the precise definition of a $\delta$-Corona decomposition, which is somewhat cumbersome to place here without first discussing the Christ-David dyadic lattice in Section \ref{subsec:lattice}.

\begin{thm} \label{thm:main}
Fix $n, d \in \N$ with $0 < d <n$ and $C_E >0$. Then there are constants $\delta_0 >0$ and $\theta_0 \in (0,1)$ depending only on $n$, $d$, and $C_E >0$, so that the following holds.
Whenever $ 0 < \delta < \delta_0$, $E \subset \R^n$ is $d$-Ahlfors regular with constant $C_E$, and one of the following conditions hold
\smallskip
\begin{enumerate}[(A)]
\itemsep 0.5em
\item $E$ is $\delta$-uniformly rectifiable, \label{cond:delta_ur}
\item $E$ admits $\delta$-Corona decompositions, \label{cond:delta_cor}
\item $E$ is upper $d$-Ahlfors regular with constant $(1+\delta)$, and for any Borel $g$ satisfying $(1 + \delta)^{-1} \le g \le (1 + \delta)$, if $d\mu(x) = g(x) d\HD^d|_E(x)$, then for all $x \in E$ and $r >0$,\[ \mu(B(x,r))^{-1} \int_{B(x,r)} \int_0^r \alpha_\mu(y,t)^2 \;  \dfrac{d\mu(y)dt}{t} \le \delta, \] \label{cond:alpha_c}
\item $E$ is upper $d$-Ahlfors regular with constant $(1+\delta)$, and for all $x \in E$ and $r >0$, $b\beta_{1, E}(x,r) \le \delta$,
\label{cond:beta_1}
\item $E$ is upper $d$-Ahlfors regular with constant $(1+\delta)$ and $\delta$-Reifenberg flat, \label{cond:reif}
\item $E$ is $d$-rectifiable, lower $d$-Ahlfors regular with constant $(1+\delta)$, and for every $x \in E$ and $r >0$, there is a $V \in G(n,d)$ so that \[ \fint_{B(x,r)} \norm{ \pi_{T(x)} - \pi_V } \; d\HD^d|_E(x)  + \sup_{y \in B(x,r) \cap E} \dfrac{\abs{\pi_{V^\perp}(y-x)}}{r} \le \delta , \] \label{cond:bmo}
\end{enumerate}
then all of the others also hold with constant $\delta^{\theta_0}$ in place of $\delta$.
\end{thm}

\begin{rmk}[Sharpness of the Ahlfors regularity assumption]\label{rmk:adr_nec}
Let us discuss the sharpness of the small-constant Ahlfors regularity assumptions appearing in the conditions \ref{cond:delta_ur}-\ref{cond:bmo} in Theorem \ref{thm:main}. We reminder the reader that in the statement of Theorem \ref{thm:main} and the discussion that follows below, our sets $E \subset \R^n$ are always assumed to be $d$-Ahlfors regular with (large) constant $C_E> 1$.

First, in the places they appear in \ref{cond:beta_1}, \ref{cond:reif}, and \ref{cond:bmo}, they are necessary. We shall see shortly that \ref{cond:beta_1} and \ref{cond:reif} are easily seen to be equivalent. In \ref{cond:reif}, the upper Ahlfors regularity assumption can be seen to be necessary by example of a very flat snowflake, as in \cite{DTREIF}. The key point is that there are $\delta$-Reifenberg flat snowflakes for arbitrarily small $\delta$ that have infinite $\HD^d$ measure. Finite truncations of such constructions yield very flat $d$-Ahlfors regular sets $E$ with large constant $C_E \gg 1$, but for which small-constant Ahlfors regularity fails. By Lemma \ref{lemma:delta_ur_adr}, such sets are not $\delta^{\theta_0}$-UR. In \ref{cond:bmo}, one can see the lower $d$-Ahlfors regularity assumption is necessary by taking $E = V^+$ for some half $d$-plane $V^+$. Again, such an $E$ is $d$-Ahlfors regular with large constant and satisfies the other condition in \ref{cond:bmo} trivially with $\delta=0$, but is not $\delta^\theta$-UR.

This brings us to \ref{cond:alpha_c} which is more delicate. If one instead considers the measure $d\mu(x) = g(x) d\HD^d|_E(x)$ where $1/2\le g \le 2$, $g$ attains the values $1/2$ and $2$ somewhere, yet $\norm{g}_{\BMO} = \delta$, then in fact our arguments will show that still the Carleson condition 
\begin{align*}
\mu(B(x,r))^{-1} \int_{B(x,r)} \int_0^r \alpha_\mu(y,t)^2 \;  \dfrac{d\mu(y)dt}{t} \le \delta^{\theta_0},
\end{align*}
holds whenever $E$ is $\delta$-UR (see the proofs of Lemmas \ref{lemma:osc} and \ref{lemma:alpha_density}). On the other hand, $\mu$ is not $d$-Ahlfors regular with small-constant, and thus a small-constant Carleson condition on the coefficients $\alpha_\mu$ alone cannot imply small-constant Ahlfors regularity of the \textit{measure} $\mu$. This is not to say that the implication cannot hold for the measure $\HD^d|_E$, and indeed there is a subtle but important difference between $\mu$ and $\HD^d|_E$ in the $\alpha$ coefficients. At this stage we do not know whether the small-constant $d$-Ahlfors regularity assumption in \ref{cond:alpha_c} is necessary.
\end{rmk}

We prove Theorem \ref{thm:main} one step at a time, proving (in alphabetical order) each of the conditions \ref{cond:delta_ur}-\ref{cond:bmo} with constant $\delta$ implies the subsequent condition with constant $C_0 \delta^{\theta_0}$, where $C_0, \theta_0 > 0$ depend only on $n$, $d$, and $C_E$.
Instead of repeating this phrase over and over, we shall instead write ``\ref{cond:delta_ur} gives \ref{cond:delta_cor}'', when really we mean that \ref{cond:delta_ur} implies \ref{cond:delta_cor} with constant $C_0 \delta^{\theta_0}$ in place of $\delta$. 
By taking $\delta_0$ and $\theta_0$ even smaller, this is enough to prove the Theorem.
We do not explicitly compute $\theta_0$ in the proof of each implication, except for where there is a clear optimal power; instead, we care only that each condition is quantitatively controlled by the previous one. 

The bulk of our work (and our main contribution) is in showing \ref{cond:delta_ur} gives \ref{cond:delta_cor}, and \ref{cond:delta_cor} gives \ref{cond:alpha_c}, which are done in Sections \ref{sec:ur_cor} and \ref{sec:cor_alpha} respectively.
Here we should emphasize as stated previously that for (large-constant) Ahlfors regular, uniformly rectifiable measures $\mu$, one has the large-constant Carleson measure estimate
\begin{align*}
\sup_{x \in \spt \mu, r >0} \mu(B(x,r))^{-1} \int_{B(x,r)}\int_0^r \alpha_\mu(y,t)^2 \; \frac{d\mu(y)dt}{t} < \infty,
\end{align*}
(see \cite[Theorem 1.2]{TOLSAALPHA}). However, it takes delicate analysis to show that this quantity is small for $\delta$-UR sets (and in fact, necessitates the appropriate notion of a ``small-constant''-Corona decomposition, which we introduce here).

That \ref{cond:alpha_c} gives \ref{cond:beta_1} is immediate once one recalls the fact that for Ahlfors regular sets $E$, the Carleson measure estimate in \ref{cond:alpha_c} in fact implies that $\alpha_\mu(x,r)^2 \le C \delta$, and that the $\alpha_\mu$ dominate the $b\beta_{1, E}$ \cite[Lemma 3.2]{TOLSAALPHA}. Similarly that \ref{cond:beta_1} gives \ref{cond:reif} is a straight-forward estimate that uses the Lipschitz nature of the distance function and $d$-Ahlfors regularity of $E$ to show $b\beta_{\infty,E}(x,r) \le C b\beta_{1,E}(x,r)^{1/(d+1)}$. As such we omit the proofs.
We prove that \ref{cond:reif} gives \ref{cond:bmo} in Section  \ref{sec:reif_gamma} from an argument that estimates the portion of $E$ whose tangent planes make a large angle with a good approximation to $E$ in a ball $B(x,r)$. This argument is different than the proof using the Gauss-Green Theorem by Kenig and Toro in co-dimension $1$ (see Theorem 2.1 in \cite{KT97} and also the proof following (2.18) in \cite{BEGTZ22}), and in particular also works in any co-dimension.
Finally, the proof of \ref{cond:bmo} gives \ref{cond:delta_ur} exists in several forms in the literature. When $d =n-1$ and $E$ is a smooth enough hypersurface, the argument is due to Semmes \cite[Proposition 5.1]{CASI} (and the resulting approximating Lipschitz graphs are referred to as Semmes decompositions) and later used in \cite{KT97}. It is also proved under different topological assumptions of a domain $\Omega$ in \cite[Theorem 4.16]{HMT10}, where the hypothesis on the quantity $\abs{\pi_{V^\perp}(y-x)}/r$ is removed, and proved by other means. When $d < n -1$, this implication is essentially proved in \cite[Lemma 3.2]{BLATT} again when $E \subset \R^n$ is a $C^1$ manifold, but for the sake of completeness, we outline the proof of Blatt in Section \ref{sec:gamma_ur} to make clear the fact that in our setting (and with the Ahlfors regularity assumptions), the argument does not require $E$ to be a $C^1$ manifold.

\subsection{An application to chord-arc domains}\label{subsec:intro_cad}

Let us end the introduction with a discussion relating $\delta$-UR sets of dimension $(n-1)$ in $\R^n$, and $\delta$-chord-arc domains (as defined in \cite{KT99}), as promised earlier. All of the arguments involved in the proof of Theorem \ref{thm:main} are local, and thus there are corresponding local and ``vanishing'' results that follow from these arguments, though they are slightly technical to write down. In fact, these local results, which we leave to Section \ref{sec:local}, are in more direct analogy to the so-called $\delta$-chord-arc domains introduced by Kenig in Toro in \cite{KT97} and \cite{KT99}. Let us define these rigorously now. 

\begin{defn}\label{defn:sep}
A domain $\Omega \subset \R^n$ is said to satisfy the separation property if for each $K \subset \R^n$ compact, there is an $R_K >0$ so that for each $x \in \partial \Omega \cap K$ and each $r \in (0, R_K)$, there is a choice of $V \in A(n,n-1)$ and choice of normal vector $\vec{n}_V$ to $V$ so that $x \in V$, and
\begin{align*}
\mathcal{T}^+(r, x) =&  \{  y + t \vec{n}_V \in B(x,r) \; : \; y \in V, \; t > r/4\} \subset \Omega, \\
\mathcal{T}^-(r, x) = & \{  y + t \vec{n}_V \in B(x,r) \; : \; y \in V, \; t < r/4\} \subset \Omega^c.
\end{align*}
If $\Omega$ is unbounded, we assume also that $\partial \Omega$ divides $\R^n$ into two distinct, nonempty connected components.
\end{defn}

\begin{defn}\label{defn:rf_domain}
Let $\delta \in (0, \delta_n)$ for some small dimensional constant $\delta_n >0$. A domain $\Omega \subset \R^n$ is said to be a $\delta$-Reifenberg flat domain if for each $K \subset \R^n$ compact, there is an $R_K > 0$ so that for each $x \in \partial \Omega \cap K$ and each $r \in (0, R_K)$, $b\beta_{\infty, \partial \Omega}(x,r) \le \delta$. If $\Omega$ is unbounded, we assume also that 
\begin{align*}
\sup_{x \in \partial \Omega, \; r > 0 } b\beta_{\infty, \partial \Omega}(x,r) \le \delta_n.
\end{align*}
\end{defn}

\begin{defn}\label{defn:cad}
Let $\delta \in (0 ,\delta_n)$. A set of locally finite perimeter $\Omega \subset \R^n$ is called a $\delta$-chord-arc domain if $\Omega$ is a $\delta$-Reifenberg flat domain satisfying the separation property, $\partial \Omega$ is $(n-1)$-Ahlfors regular, and in addition the following holds. For each $K \subset \R^n$ compact, there is an $R_K >0$ so that for each $x \in \partial \Omega \cap K$,
\begin{align*}
\norm{\vec{n}}_*(B(x,R_K)) \le \delta.
\end{align*}
Here $\vec{n}(y)$ is the unit outer normal to $\partial \Omega$, $\vec{n}_{y,s} = \fint_{\partial \Omega \cap B(y,s) } \vec{n}(z) \; d\HD^{n-1}(z)$, and  
\begin{align*}
\norm{\vec{n}}_*(B(y,r)) = \sup_{0 < s \le r} \left( \fint_{\partial \Omega \cap B(y,r)} \abs{\vec{n}(z) - \vec{n}_{y,s}}^2 \; d\HD^{n-1}(z) \right)^{1/2}.
\end{align*} 
\end{defn}

In the terminology we have introduced thus far, there is no immediate containment between $\delta$-UR sets of dimension $(n-1)$ and boundaries of $\delta$-chord arc domains. This is because chord-arc domains satisfy topological separation conditions and are sets of locally finite perimeter, and because $\delta$-UR sets satisfy global flatness conditions, while $\delta$-chord-arc domains satisfy local ones.  However, these differences are minor, and the two notions are very closely related. In particular, equation (2.18) in \cite{BEGTZ22} says that for $\delta$-chord arc domains, one has $\abs{\langle \vec{n}_{x,r}, y-x \rangle } \le C \delta^{1/2} r$ for $y \in B(x,r)$ whenever $\norm{\vec{n}}_* (B(x,r)) \le \delta$. This implies that the second condition in \ref{cond:bmo} holds locally for $\delta$-chord arc domains. In addition one can prove local lower $(n-1)$-Ahlfors regularity of $\partial \Omega$ (with small constant) from the local Reifenberg flatness condition (see the proof of Theorem \ref{thm:reif_bmo}). Combining these with the fact that the proof of Theorem \ref{thm:main} is local, we see that when $\Omega$ is a $\delta$-chord-arc domain, then on compact sets for small enough scales, $\partial \Omega$ satisfies the $\delta^{\theta_0}$-UR conditions. This is made precise by the following Theorem, and in fact, as long as we assume some underlying conditions on a domain $\Omega$, we obtain a new characterization of $\delta$-chord-arc domains. For simplicity, we choose just one such condition coming from \ref{cond:delta_ur}-\ref{cond:bmo} to give the characterization, which we make as the following local definition.
\begin{defn}\label{defn:local_carl}
Let $\mu$ be $d$-Ahlfors regular. We say that $\mu$ satisfies the local $\delta$-UR condition of dimension $d$ if for each $K \subset \R^n$ compact, there is an $R_K >0$ so that for each $x \in \spt \mu \, \cap K$ and $r \in (0,R_K)$ one has $\mu(B(x,r)) \le (1 + \delta)r^d$ and 
\begin{align*}
\mu(B(x,r))^{-1} \int_{B(x,r)}\int_0^r \alpha_{\mu}(y,t)^2 \; \dfrac{ d\mu(y) dt}{t} \le \delta.
\end{align*}
\end{defn}

\begin{thm}\label{thm:cad}
Fix $n \in \N$ and $C_E > 0$. Then there are constants $\delta_0, \theta_0 \in (0,1)$ depending only on $n$ and $C_E$ so that the following holds. 

Let $\Omega \subset \R^n$ be a set of locally finite perimeter such that $\Omega$ satisfies the separation property and $\partial \Omega$ is $(n-1)$-Ahlfors regular with constant $C_E$. If $\Omega$ is unbounded, assume in addition that $\sup_{x \in \partial \Omega, \; r > 0 } b\beta_{\infty, \partial \Omega}(x,r) \le \delta_0$. 
Then for any $\delta \in (0, \delta_0)$ each of the conditions 
\smallskip
\begin{enumerate}[(I)]
\itemsep 0.5em
\item $\Omega$ is a $\delta$-chord arc domain, \label{cond:cad}
\item for any measure $d\mu(x) = g(x) \HD^{n-1}|_{\partial \Omega}(x)$ with $(1 +\delta)^{-1} \le g \le 1 + \delta$, $\mu$ satisfies the local $\delta$-UR condition of dimension $(n-1)$, \label{cond:loc_carl}
\end{enumerate}
implies the other with constant $\delta^{\theta_0}$ in place of $\delta$. 
\end{thm}
For a discussion of the proof of Theorem \ref{thm:cad}, see Section \ref{sec:local}.

\section{Preliminary definitions} \label{sec:prelim}
We introduce the system of ``dyadic cubes'' for Ahlfors regular sets, which is an integral part of the definition of a $\delta$-Corona decomposition. They also play an important role in the square function estimates we prove on the Tolsa $\alpha$ coefficients in Theorem \ref{thm:corona_alpha}, since we opt to prove a dyadic version instead of the continuous one.
\subsection{The Christ-David dyadic lattice} \label{subsec:lattice} 
Recall that as in \cite{DWAVELETS}, if $E$ is a $d$-Ahlfors regular set in $\R^n$ with constant $C_E$, then one can construct a family of subsets of $E$ that plays an analogous role to the family of dyadic cubes in $\R^n$.
In particular, for each $j \in \Z$, there is a partition $\Delta_j$ of $E$ into ``dyadic cubes" of $E$ that satisfy the following:
\begingroup
\addtolength{\jot}{0.5em}
\begin{align}
& \text{ \parbox[t]{0.8 \textwidth}{ if $j \le k, Q \in \Delta_j,$ and $Q' \in \Delta_k$, then either $Q \cap Q' = \emptyset$ or $Q \subset Q',$} }  \label{cond:dyadic1} \\
& \text{ \parbox[t]{0.8 \textwidth} {if $Q \in \Delta_j$, then $\cd^{-1} 2^j \le \diam Q \le \cd 2^j$ and $\cd^{-1} 2^{jd} \le \HD^d(Q) \le \cd 2^{jd},$}}  \label{cond:dyadic2} \\
& \text{ \parbox[t]{0.8 \textwidth}{if $Q \in \Delta_j$ and $\tau > 0$ then $\HD^d \left( \left \{ x \in Q \; : \; \dist(x, E \setminus Q) \le \tau 2^j \right \} \right ) \le \cd \tau^{1/\cd} 2^{jd}.$}} \label{cond:dyadic3}
\end{align}
\endgroup
Remark that in (\ref{cond:dyadic1})-(\ref{cond:dyadic3}) above, the constant $\cd$ only depends on the dimensions $n, d$ and $C_E$.
Also, condition (\ref{cond:dyadic3}) for $\tau$ sufficiently small furnishes the existence of a ``center" of each cube $c_Q \in Q$, which satisfies 
\begin{align*}
\dist(c_Q, E \setminus Q) \ge \cd^{-1} \diam Q,
\end{align*}
so that 
\begin{align}
B(c_Q, \cd^{-1} \diam Q) \cap E \subset Q. \label{cond:dyadic4}
\end{align}

By convention, we define for $\lambda >1$,
\begin{align*}
\lambda Q = \{x \in E \; : \; \dist(x, Q) \le (\lambda - 1) \diam Q\}.
\end{align*}
In a similar manner, for any $Q$, we define $B_Q = B(c_Q, \diam Q)$ so that $B_Q$ is a ball centered on $E$ satisfying 
\begin{align*}
Q \subset B_Q \cap E \subset 2Q. 
\end{align*}
If $Q \subset Q'$, and $Q \in \Delta_j, Q' \in \Delta_{j+1}$ then $Q$ is said to be a child of $Q'$, and $Q'$ is said to be the parent of $Q$.
Similarly, if $R$ and $R'$ share a parent, then they are said to be siblings.
The set of all dyadic cubes of $E$ is $\Delta = \cup_j \Delta_j$, and for $R \in \Delta$, we denote all dyadic cubes contained in $R$ by $\Delta(R)$.
Finally, if $Q \in \Delta$ belongs to $\Delta_j$, we write $\gen Q = j$.

\subsection{Small-constant Corona Decompositions}
Let us define precisely $\delta$-Corona decompositions for $d$-Ahlfors regular sets.
We opt to make the definition as strong as possible, since we anticipate this will be the most useful property of small-constant UR sets from which one can obtain precise, quantitative, small constant square function estimates. Fix a constant $C_{n,d} > 1$ large, and to be determined below in the discussion of Remark \ref{rmk:corona}.

\begin{defn} \label{defn:corona}
Suppose that $E \subset \R^n$ is $d$-Ahlfors regular, and $E$ has a system of dyadic cubes $\Delta$ with constant $\cd \le C_{n,d}$. Then we say that $E$ admits a $\delta$-Corona decomposition in $R_0 \in \Delta$ if for each $R \in \Delta$ such that $\gen R = \gen R_0$ and $R \subset \delta^{-1} B_{R_0}$, there is a partition $\family(R)$ of $\Delta(R)$ which satisfies the following:
\begingroup
\addtolength{\jot}{0.5em}
\begin{align}
& \text{ \parbox[t]{0.8 \textwidth}{ each $S \in \family(R)$ has a maximal cube $Q(S)$ so that if $Q \in S$ and some $Q' \in \Delta(R)$ satisfies $Q \subset Q' \subset Q(S)$, then $Q' \in S$. Moreover, if $Q \in S$, then either all of its children, or none of its children are in $S$.} } \label{cond:corona_coherent}\\
& \text{ \parbox[t]{0.8 \textwidth} { for each $S \in \family(R)$ there is a $d$-dimensional Lipschitz graph $\Gamma= \Gamma(S)$ with Lipschitz constant $\le \delta$ so that for every $Q \in S$, $$ \HD^d(\delta^{-1} B_{Q(S)} \cap (E \Delta \Gamma)) \le \delta \HD^d(Q(S)).$$ Moreover,  for each $Q \in S$ and $x \in \delta^{-1} B_Q \cap ( E \cup \Gamma)$, we have $\dist(x, E) + \dist(x, \Gamma) \le \delta \diam Q$. (Here $E\Delta \Gamma = (E \setminus \Gamma) \cup (\Gamma \setminus E)$ is not to be confused with the symbol for dyadic lattices $\Delta = \cup_{j \in \Z} \Delta_j$).  }}   \label{cond:corona_graph} \\
& \text{ \parbox[t]{0.8 \textwidth} { the maximal cubes $Q(S)$ satisfy a small-constant Carleson packing condition in that for each $R' \in \Delta(R)$, one has $$\sum\limits_{\substack{S \in \family(R) \\ Q(S) \subset R'}}  \HD^d(Q(S)) \le (1 + \delta) \HD^d(R').$$ }}  \label{cond:corona_carleson} 
\end{align}
Moreover, we have the following condition on the ``top'' Lipschitz graphs of the Corona decomposition:
\begin{align}
& \text{ \parbox[t]{0.8 \textwidth}{ If $S_R, S_{R_0}$ are so that $R \in S_R \in \family(R)$ and $R_0 \in S_{R_0} \in \family(R_0)$, then $\Gamma(S_R) = \Gamma(S_{R_0})$. That is to say, each of the chosen Lipschitz graphs for the collections containing the top cubes $R$ are identical.} } \label{cond:corona_topgr}
\end{align}
\endgroup
Finally we say that $E$ admits $\delta$-Corona decompositions if it admits a $\delta$-Corona  decomposition in each $R_0 \in \Delta$. 
\end{defn}

\begin{rmk} \label{rmk:corona}
Some remarks about this definition (and how it is different from the usual Corona decomposition as in \cite{DSSIO}) are in order.
In general, a Corona decomposition for a uniformly rectifiable set $E$ includes a partition of dyadic cubes into ``good cubes" and ``bad cubes," where the bad cubes do not have a good approximating Lipschitz graph as in (\ref{cond:corona_graph}).
In the small-constant setting, it turns out that all cubes are ``good," and thus the main condition satisfied is that there are not too many families $S \in \family(R)$ as quantified by (\ref{cond:corona_carleson}).
Also, since we are interested in bi-lateral approximations of Ahlfors regular sets by planes, we include in (\ref{cond:corona_graph}) that the approximating graph $\Gamma$ be sufficiently close to $E$ as well.
The facts that there are measure estimates on $E \Delta \Gamma$ inside $\delta^{-1} B(Q(S))$ and that (\ref{cond:corona_topgr}) holds are perks we obtain for free when showing \ref{cond:delta_ur} gives \ref{cond:delta_cor}, which shall be useful to us in estimating the Tolsa $\alpha$ coefficients for $\delta$-UR sets. However, it should be noted that these measure estimates in condition \eqref{cond:corona_graph} are the strongest; for $\delta$ sufficiently small, the condition \eqref{cond:corona_graph} implies that $E$ is $C\delta$-UR (recall Definition \ref{defn:delta_ur}). Since we shall use this fact later, we provide a quick proof in Lemma \ref{lem:cor_ur} below.

One other main difference is that in a $\delta$-Corona decomposition, as opposed to a general one, we require that the Carleson packing constant appearing in (\ref{cond:corona_carleson}) be controlled as $\delta \da 0$.
This plays a crucial role in the arguments that follow, since this implies that if $R' \in \Delta(R)$ is a maximal cube in some family, $R' = Q(S)$ for some $S \in \family(R)$, then necessarily one has 
\begin{align}
\sum\limits_{\substack{S \in \family(R) \\ Q(S) \subsetneq R'}} \HD^d(Q(S)) \le \delta \HD^d(R'). \label{cond:corona_carleson_cor}
\end{align}

Also it is important to remark that the so-called ``coherent" condition on the families $\family(R)$ from (\ref{cond:corona_coherent}) includes two pieces.
The second part, which asserts that if $Q \in S$ then either all of its children are or none of its children are, has as an important consequence that 
\begin{align}
& \text{ \parbox[t]{0.8 \textwidth}{  if $x \in Q(S)$ then either $x$ is in arbitrarily small cubes of $S$, or $x$ is contained in a minimal cube of $S$. }} 
\end{align}

Finally this brings us to the appearance (and definition) of the constant $C_{n,d}$. Notice that by definition, a $\delta$-Corona decomposition of $E$ is assumed to hold over a dyadic system $\Delta$ with bounded constant $\cd \le C_{n,d}$ where $C_{n,d}$ is chosen as follows. We shall soon see (Lemma \ref{lemma:delta_ur_adr}) that $\delta$-UR sets of dimension $d$ in $\R^n$ are $d$-Ahlfors regular with uniformly bounded constant. In particular, they admit a system of dyadic cubes $\Delta$ as in Section \ref{subsec:lattice} with constant $\cd$ depending only $n, d$, which we define to be $C_{n,d}$. Forcing this condition on the system $\Delta$ is rather minor, but it allows us to rule out pathological examples of $\delta$-Corona decompositions for $\delta$-UR sets associated to a system of dyadic cubes with very large constant.
\end{rmk}

\begin{lemma}\label{lem:cor_ur}
Suppose that $E$ is $d$-Ahlfors regular and admits a system of dyadic cubes with constant $C_D >1$. Then if $E$ admits $\delta$-Corona decompositions for $0 < \delta < (2 C_D +1)^{-1}$ and $\delta$ sufficiently small, then $E$ is $(C_D \delta)$-UR.
\end{lemma}
\begin{proof}
Fix $x \in E$ and $r >0$. Choose some integer $j \in \Z$ so that $2^j \le r \le 2^{j+1}$. Since $\Delta_j$ is a partition of $E$, we may choose a dyadic cube $R_0 \in \Delta_j$ so that $x \in R_0$. 

Now since $E$ admits $\delta$-Corona decompositions, then it admits a $\delta$-Corona decomposition in $R_0$, and thus there is a partition of $\Delta(R_0)$ into coherent subfamilies $S \in \mathcal{F}(R_0)$, $S \subset \Delta(R_0)$ satisfying conditions  \eqref{cond:corona_coherent}-\eqref{cond:corona_topgr}. Since the subfamilies $S \in \mathcal{F}(R_0)$ partition $\Delta(R_0)$ and $R_0 \in \Delta(R_0)$, there is some $S \in \mathcal{F}(R_0)$ so that $R_0 \in S$. Notice that the maximal cube $Q(S)$ of $S$ is a subset of $R_0$, but since $R_0 \in S$ then $R_0 \equiv Q(S)$.

Applying condition \eqref{cond:corona_graph} to $R_0 \equiv Q(S)$, we see that there is Lipschitz graph $\Gamma$ with Lipschitz constant $\le \delta$ so that 
\begin{align}\label{eqn:goal}
\HD^d(\delta^{-1} B_{R_0} \cap (E \Delta \Gamma)) \le \delta \HD^d(R_0) \le \delta C_D 2^{jd} \le \delta C_D r^d.
\end{align}
By \eqref{eqn:goal}, the proof will be finished as long as we show that $\delta^{-1}B_{R_0} \supset B_r(x)$. However, if $\abs{y-x} < r$, then
\begin{align*}
\abs{y-c_{R_0}} \le \abs{y- x} + \abs{x-c_{R_0}} < r + \mathrm{diam}(R_0) \le 2^{j+1} + \mathrm{diam} (R_0) \le (2 C_D + 1) \mathrm{diam}(R_0),
\end{align*}
where $c_{R_0}$ is the center of $R_0$. This proves the inclusion, since $B_{R_0} \coloneqq B_{\mathrm{diam}(R_0)}(c_{R_0})$.
\end{proof}

\subsection{Conventions for constants}
In general, we denote by $C$ a constant which is allowed to change line per line, depending on the parameters explicitly stated in the statement of a Lemma, Theorem, or Corollary. We avoid using the symbols $\lesssim, \gtrsim$ but very infrequently will use the notation $A \simeq_D B$ to mean that there is some constant $C >0$ depending only on $D$ so that $C^{-1}A \le B \le C A$.

\section{\texorpdfstring{$\delta$}{delta}-UR measures admit \texorpdfstring{$\delta^{\theta_0}$}{small}-Corona decompositions} \label{sec:ur_cor}
In this section, we show that \ref{cond:delta_ur} gives \ref{cond:delta_cor}, i.e., we prove Theorem \ref{thm:ur_implies_corona}.
Let us begin within two useful lemmas.
The first shall be used repeatedly in future Sections.

To motivate the first result, remark that if $\mu$ is $d$-Ahlfors regular with constant $\cmu> 0$ and support $E$, then in general we may only conclude that $\mu$ and $\HD^d|_E$ are mutually absolutely continuous with density $g = d \mu/d\HD^d|_E$ satisfying $\cmu^{-1} \le g \le 2^d \cmu$, and moreover, $E$ is $d$-Ahlfors regular with constant $2^d \cmu^2$ (see for example, \cite[Theorem 6.9]{Mattila}).
This crude estimate is problematic if we want precise control of the Ahlfors regularity constant of $\HD^d|_E$ when $\mu$ is $d$-Ahlfors regular with constant $\cmu$ that is close to $1$.
When more geometric regularity is assumed, though, this can be strengthened as in the following.

\begin{lemma}\label{lemma:adr_m_adr_set}
Suppose that $\mu$ is a $d$-Ahlfors regular measure in $\R^n$ with constant $\cmu>0$, and that $E = \spt \mu$ is $d$-rectifiable.
Then the density $d\HD^d|_E/d\mu$ exists and satisfies $$C_\mu^{-1} \le d\HD^d|_E/d\mu \le C_\mu,$$ $\mu$-almost everywhere.
In particular, for any subset $A \subset \R^n$ Borel, we have 
\begin{align}
\cmu^{-1} \le \dfrac{\HD^d|_E(A)}{ \mu(A)} \le \cmu, \label{eqn:adr_m_adr_set}
\end{align}
and $\HD^d|_E$ is $d$-Ahlfors regular with constant $\cmu^2 >0$. 
\end{lemma}
\begin{proof}
It is straight-forward to see from the Ahlfors regularity of $\mu$ that $E$ is also $d$-Ahlfors regular, and $\HD^d|_E$ and $\mu$ are mutually absolutely continuous with density $d\HD^d|_E/d\mu$ bounded above and below.
Since $E$ is rectifiable, we know that the density \[\theta^d(x) = \lim_{r \downarrow 0}  \HD^d|_E(B(x,r))/ r^d \] exists and equals $1$ for $\HD^d$ almost all $x \in E$ (see, for example, Theorem 16.2 in \cite{Mattila}).
It follows then that for $\HD^d|_E$ (and thus $\mu$) almost all $x$, 
\begin{align*}
\dfrac{d \HD^d|_E}{d \mu} (x) & \le \limsup_{r \da 0} \dfrac{\HD^d|_E(B(x,r))}{ \mu(B(x,r))} \\
& = \limsup_{r \da 0} \dfrac{\HD^d|_E(B(x,r))} {r^d} \dfrac{r^d}{\mu(B(x,r))} \\
& \le \cmu,
\end{align*}
by Ahlfors regularity of $\mu$.
A similar computation shows $d\HD^d|_E/d\mu(x) \ge \cmu^{-1}$ for $\mu$ almost all $x$, and thus whenever $A \subset \R^n$ is Borel,
\begin{align*}
\cmu^{-1} \mu(A)  = \int_{A} \cmu^{-1} \; d\mu  \le \int_A \dfrac{d\HD^d|_E}{d\mu} \; d\mu  \le \cmu  \int_A \; d\mu  = \cmu \, \mu(A). 
\end{align*}
Since $\HD^d|_E(A) = \int_A (d \HD^d|_E /d \mu) \; d\mu$ (see for example, Theorem 2.12 in \cite{Mattila}), this shows \eqref{eqn:adr_m_adr_set}. The last claim of the Lemma follows by taking $A = B(x,r)$ for $x \in E$ in \eqref{eqn:adr_m_adr_set} and using Ahlfors regularity of $\mu$. 
\end{proof}

The proof of the following lemmas are omitted, since they are proved (when $d = n-1$) in \cite[Lemma 5.4]{DLMAinfty}. The arguments in higher codimension are the same.
\begin{lemma}[see Lemma 5.4, \cite{DLMAinfty}]\label{lemma:delta_ur_adr}
There is a constant $C_0 \ge 1$ depending only on $n$ and $d$ such that if $E \subset \R^n$ is $\delta$-UR of dimension $d$, then $E$ is $d$-Ahlfors regular with constant at most $1 + C_0 \delta^{1/d}$ and $C_0 \delta^{1/d}$-Reifenberg flat.
\end{lemma}

We may now state the main Theorem of this section, which is that \ref{cond:delta_ur} gives \ref{cond:delta_cor}. 
Of course, in our setting of $\delta$-UR sets, life becomes easier in that we need not go through the effort of constructing the Lipschitz graphs \emph{by hand} in a small-constant Corona decomposition, as the authors do in \cite{DSSIO}.
Instead, the following result simply says that with our approximating Lipschitz graphs coming from the definition of a $\delta$-UR set, we obtain a Corona decomposition with a loss in a constant, and an exponent in $\delta$. 

\begin{thm}\label{thm:ur_implies_corona} There are $C_0 \ge 1$, $\delta_0 >0$ and $\theta_0 \in (0,1)$ depending only on the underlying dimensions $n$ and $d$ so that if $E \subset \R^n$ is $\delta$-UR of dimension $d$ in $\R^n$ with $\delta \in (0 ,\delta_0)$, then $E$ admits $C_0 \delta^{\theta_0}$-Corona decompositions.
\end{thm}
\begin{proof}
As mentioned at the end of Remark \ref{rmk:corona}, since $E$ is $\delta$-UR of dimension $d$, we may fix once and for all a system of dyadic cubes $\Delta$ for $E$ with constant $\cd \le C_{n,d}$.

Now begin with some dyadic cube $R_0 \in \Delta$ for $E$.
Denote by $C_E$ the Ahlfors regularity constant of $\HD^d|_E$, which by Lemma \ref{lemma:delta_ur_adr}, is bounded.
Fix $\theta, \theta' \in (0,1)$ to be determined, and set $\eta \coloneqq \delta^\theta, M \coloneqq \delta^{-\theta'}$.
For definiteness, we state now that $\theta = 1/2$ and $\theta' < \theta/(2d)$ shall suffice here, but these parameters are different than the $\theta_0$ in the conclusion of the Theorem.
We construct our partition $\family \equiv \family(R_0)$ by sequential coherent generations.
That is, we will construct $\family$ as a disjoint union $\family = \family_0 \cup \family_1 \cup \family_2 \cup \dotsc$ where each $\family_i$ consists of coherent collections $S \subset \Delta(R_0)$, and $\family_0$ contains a single collection $S_0$ with top cube $Q(S_0) = R_0$.
Moreover, each $\family_i$ for $i \in \N$ will satisfy the following: for each $S \in \family_i$, there is a unique $S' \in \family_{i-1}$ so that $Q(S) \subset Q(S')$, and all cubes $Q \in \Delta(R_0)$ with $Q(S) \subsetneq Q \subset Q(S')$ are such that $Q \in S'$.
Also, we shall deal only with the partition $\family = \family(R_0)$ of $\Delta(R_0)$ for now, and leave to the very end of the proof how to ensure that (\ref{cond:corona_topgr}) holds for the other $R \in \Delta$ that are nearby $B_{R_0}$.

As mentioned, we shall take $R_0$ to be the top cube of the only collection $S_0$ in the zeroth generation family, $\family_0$.
Since $E$ is $\delta$-UR, we choose a $\delta$-Lipschitz graph $\Gamma \equiv \Gamma(S_0)$ so that 
\begin{align}
\HDE(M B_{R_0} \setminus \Gamma) + \HD^d|_\Gamma(MB_{R_0} \setminus E) & \le \delta ( M  \diam R_0)^d.  \nonumber \\
& = \delta^{1 - d \theta'} (\diam R_0)^d. \label{eqn:ur_cor_1}
\end{align}
In what follows, estimate (\ref{eqn:ur_cor_1}) (and the fact that $\Gamma, E$ are sufficiently flat) shall be the \textit{only} fact we use about $\Gamma$ to ensure that this particular $\Gamma$ shall suffice in the construction of $S_0$.

Now we continue adding children of $R_0$ to the collection $S_0$ until we reach a cube $Q$ that has a sibling $Q'$ (possibly $Q' = Q$) which is mediocre for $S_0$, meaning that
\begin{align}
\HD^d(Q' \setminus \Gamma) > \eta \HD^d(Q'). \label{eqn:ur_cor_2}
\end{align}
At this stage, $Q$, and all of its siblings become minimal cubes of the collection $S_0$, and all of their children become top cubes for the new collections $S \in \family_1$.
Notice that for such $Q$, if $\tilde{Q}$ is the parent of $Q$, then $\tilde{Q}$ is not mediocre for $S_0$, and thus, 
\begin{align}
\HD^d(Q \setminus \Gamma) & \le \HD^d( \tilde{Q} \setminus \Gamma )   \le \eta \HD^d(\tilde{Q})  \le C \eta \HD^d(Q), \label{eqn:ur_cor_4}
\end{align}
with constant $C$ depending only on $\cd$ and the underlying dimensions.

Let us make some observations about the family $S_0$ constructed.
First of all, $S_0$ is coherent (i.e., satisfies condition (\ref{cond:corona_coherent})) by construction.
Moreover, from (\ref{eqn:ur_cor_4}) we see that all $Q \in S_0$ satisfy 
\begin{align}
\HD^d( Q \setminus \Gamma ) \le C \eta \HD^d(Q), \label{eqn:ur_cor_3}
\end{align}
since those $Q \in S_0$ that are not minimal satisfy the above inequality with $C = 1$.
Let us show now that this measure estimate implies that for the cubes in $S_0$, $\Gamma$ and $E$ are very near each other in that for any $Q \in S_0$, we have
\begin{align}
\sup_{x \in M B_Q \cap (E \cup \Gamma) } \dist(x, \Gamma) + \dist(x, E) & \le C M^2 \eta^{1/d} \diam Q \nonumber \\
& =C \delta^{\theta/d - 2 \theta'} \diam Q ,\label{eqn:ur_cor_5}
\end{align}
for some constant $C$ depending only on the dimensions and $C_D$. 

First, suppose that $Q \in S_0$, and $x \in B(c_Q, \cd^{-1} \diam Q/2) \cap Q \setminus \Gamma$ (recall that $c_Q$ is the center of $Q$, for which (\ref{cond:dyadic4}) holds).
Denoting $r = \dist(x, \Gamma)$, then as long as $\delta_0$ is sufficiently small, we must have $r \le \cd^{-1} \diam Q/2$.
This is because otherwise we have $Q \setminus \Gamma \supset Q \cap B(x, \cd^{-1} \diam Q /2)$, and thus $\HD^d(Q  \setminus \Gamma) \ge \HD^d(Q \cap B(x, \cd^{-1} \diam Q /2)) \ge c \HD^d(Q)$ for some constant $0 < c < 1$ depending only on $\cd$ and $C_E$.
When $\delta_0$ (and thus $\eta$) is sufficiently small, this contradicts (\ref{eqn:ur_cor_3}) and whence the fact that $Q \in S_0$.
Hence, we may assume that $r \le \cd^{-1} \diam Q/2$, so $B(x,r) \subset B(c_Q, \cd^{-1} \diam Q)$, and thus
\begin{align*}
\HD^d(Q \setminus \Gamma) & \ge \HD^d (Q \cap B(x,r))   \ge C_E^{-1} r^d.
\end{align*}
Since $Q \in S_0$, we have that (\ref{eqn:ur_cor_3}) gives $r \le C \eta^{1/d} \diam Q$, i.e.,
\begin{align}
\sup_{x \in E \cap B(c_Q, \cd^{-1} \diam Q/2)} \dist(x, \Gamma) \le C \eta^{1/d} \diam Q. \label{eqn:ur_cor_6}
\end{align}
Next, recall from Lemma \ref{lemma:delta_ur_adr} that $E$ is $C \delta^{1/d}$ Reifenberg flat, and similarly, so is $\Gamma$.
We claim that for $\delta_0$ sufficiently small, this implies
\begin{align}
\text{  \parbox[t]{0.8 \textwidth}{for every  $x \in (\Gamma \setminus E) \cap B(c_Q, \cd^{-1}\diam Q /8)$, there is a point  \\ $y \in E \cap B(x, 2 \dist(x, E)) \cap  B(c_Q, \cd^{-1} \diam Q /2)$ with   $\dist(y, \Gamma) \ge (2/3)\dist(x, E)$. }} \label{eqn:nearby_good_pt}
\end{align}
Assume \eqref{eqn:nearby_good_pt} for the time being. Along with (\ref{eqn:ur_cor_6}), the existence of such a point shows that 
\begin{align}
\sup_{x \in B(c_Q, \cd^{-1} \diam Q /8 )\cap (E \Delta \Gamma)} \dist(x, \Gamma) + \dist(x,  E) \le C \eta^{1/d} \diam Q. \label{eqn:ur_cor_7}
\end{align}
Appealing again to the fact that $E$ and $\Gamma$ are $C \delta^{1/d}$ Reifenberg-flat, one deduces (\ref{eqn:ur_cor_5}) from (\ref{eqn:ur_cor_7}) with a possibly larger $C$.
The proof is slightly technical, but it merely requires choosing good approximating planes for $E$ and $\Gamma$ at different scales.
For the sake of completeness, let us sketch a few details. 

Recall here that we use the notation $d_{x,r}$ for the normalized local Hausdorff distance as in Definition \ref{defn:hd_dist}.
In addition, $\theta, \theta'$ are such that $\theta' < \theta/(2d)$, and thus we have that $M^2 \eta^{1/d}$ can be made arbitrarily small if $\delta_0$ is chosen small enough.
Since $\Gamma$ is a $\delta$-Lipschitz graph, we know that there is some plane $d$-plane $P_\Gamma$ so that 
\begin{align*}
d_{c_Q, r}(\Gamma, P_\Gamma) \le C \delta^{1/d},
\end{align*}
as long as $\delta_0$ is sufficiently small, and as long as $r \ge \cd^{-1} \diam Q/8$.
Since $E$ is $C \delta^{1/d}$-Reifenberg flat, we may choose $d$-planes $P_E$ and $P_E'$ so that 
\begin{align*}
d_{c_Q, 2 \diam Q}(E, P_E) + d_{c_Q, 2 M \diam Q}(E, P_E') \le C \delta^{1/d}.
\end{align*} 
The fact that $P_E$ and $P_E'$ are very good approximations to $E$ inside $2 B_Q$, and run very near the center of $B_Q$, imply that $d_{c_Q, 2 \diam Q}(P_E, P_E') \le C M \delta^{1/d}$, and thus
\begin{align*}
d_{c_Q, M \diam Q}(P_E, P_E') \le C M \delta^{1/d}.
\end{align*}
Finally, recalling estimate (\ref{eqn:ur_cor_7}), we see that $d_{c_Q, \diam Q}(P_\Gamma, P_E) \le C \eta^{1/d}$, which implies $d_{c_Q, M \diam Q}(P_\Gamma, P_E) \le C \eta^{1/d}$.
Thus if we compare distances from $\Gamma$, to $P_\Gamma$, to $P_E$, then $P_E'$ and finally to $E$ inside $B(c_Q, M \diam Q)$, we obtain (\ref{eqn:ur_cor_5}).

This leaves us to justifying \eqref{eqn:nearby_good_pt}, which can be argued by contradiction. 
Indeed, if no such point $y \in E \cap B(x, 2 \dist(x, E))$ exists, then each such $y$ satisfies $\dist(y, \Gamma) < (2/3)\dist(x, E)$. 
Then the fact that $E$ and $\Gamma$ are very well approximated by $d$-planes in $B(x, 2 \dist(x, E))$ and similar arguments to those described above would lead to the contradiction that there is some $z \in E$ with $\abs{x - z} < \dist(x, E)$.
Now we simply recall that $c_Q \in E$, and so since $x \in B(c_Q, \cd^{-1} \diam Q /8)$ we have that $\dist(x, E) \le \cd^{-1} \diam Q /8$, so that necessarily, $y \in B(c_Q, \cd^{-1}\diam Q/2)$.
This completes the proof of (\ref{eqn:nearby_good_pt}).

Hence, for the first family $\family_0$, we have that the first part of (\ref{cond:corona_graph}) holds with $C M^2 \eta^{1/d} = C \delta^{\theta/d - 2 \theta'}$.
By construction (recall \eqref{eqn:ur_cor_1}) , we also have the desired measure estimate 
\begin{align}
\HD^d(MB_{Q(S_0)} \cap (E \Delta \Gamma)) \le C \delta^{1 - d\theta'} \HD^d(Q(S_0)). 
\end{align}
We have one final step for $\family_0$, which is to estimate the portion of minimal cubes of $S_0$, denoted $m(S_0)$, contained in $Q(S_0)$.

Recall that if $Q \in m(S_0)$, then necessarily $Q$ has a sibling $Q' \in m(S_0)$ which is mediocre for $S_0$, i.e., (\ref{eqn:ur_cor_2}) holds.
Then since the number of siblings of any dyadic cube in $\Delta$ is uniformly bounded (by Ahlfors regularity of $E$), we have that 
\begin{align*}
\sum_{Q \in m(S_0)} \HD^d(Q) & \le C \sum_{ \substack{ Q' \in m(S_0) \\ Q' \text{ mediocre }} } \HD^d(Q')  \le \dfrac{C}{\eta} \sum_{ \substack{ Q' \in m(S_0) \\ Q' \text{ mediocre }} }   \HD^d(Q' \setminus \Gamma ) \\
& \le \dfrac{C}{\eta} \HD^d(Q (S_0)\setminus \Gamma)  \le \dfrac{C}{\eta} \HDE(M B_{R_0} \setminus \Gamma) \\
& \le C \delta^{1 - d\theta' - \theta} \HD^d(R_0)
\end{align*}
by definition of $\eta$. Notice that if $\theta, \theta'$ are sufficiently small, then $1 - d \theta' - \theta >0$, and in particular, taking $\theta = 1/2$ and $\theta' < \theta/2d$ shall suffice for these purposes.

Now assuming that $\family_{i-1}$ has been constructed for $i \in \N$, we make each child $Q_0$ of some $Q ' \in m(S')$ for $S' \in \family_{i-1}$ a top cube $Q_0 = Q(S)$ of a new collection $S \in \family_{i}$, and construct $S \in \family_{i}$ in the same way as we did $S_0 \in \family_0$.
That is, since $E$ is $\delta$-UR, we choose a Lipschitz a $\delta$-Lipschitz graph $\Gamma = \Gamma(Q_0)$ for which 
\begin{align*}
\HDE(M B_{Q_0} \setminus \Gamma) + \HD^d|_\Gamma( M B_{Q_0} \setminus E) \le \delta (M \diam Q_0)^d.
\end{align*}
We continue to add subcubes of $Q \in \Delta(Q_0)$ to the collection $S$ until we find a some cube $Q$ who has a sibling $Q'$ which is mediocre for $S$ in that (\ref{eqn:ur_cor_2}) holds.
At this stage $Q$ and all of its siblings become minimal cubes of $S$, and each of their children become top cubes in the next generation $\family_{i+1}$.
The same proof above applies to this collection $S$ in place of $S_0$: it is coherent in that (\ref{cond:corona_coherent}) holds, and in addition, we have the following estimates:
\begin{align}
\sup_{x \in M B_Q \cap ( E \cup \Gamma)} \dist(x, E) + \dist(x, \Gamma (Q(S)) ) & \le C \delta^{\theta/d - 2\theta'} \diam Q \text{ for } Q \in S,  \\
\HD^d(M B_{Q(S)} \cap  \left(E \Delta \Gamma(Q(S))) \right ) & \le C \delta^{1 - d\theta'} \HD^d(Q(S)), \\
\sum_{Q \in m(S)} \HD^d(Q) & \le C \delta^{1-d\theta' - \theta} \HD^d(Q(S)). \label{eqn:ur_cor_8}
\end{align}
Recalling that $M = \delta^{-\theta'}$, we conclude that $E$ admits a $C \delta^{\theta_0}$ Corona decomposition with $\theta_0 = \min \{\theta/d - 2\theta', 1 -d\theta' -\theta,  \theta'  \} > 0$ provided that we show  (\ref{eqn:ur_cor_8}) implies (\ref{cond:corona_carleson}) with $C \delta^{\theta_0}$ in place of $\delta$, i.e.,
\begin{align}
\sum_{\substack{S \in \family \\ Q(S) \subset R }} \HD^d(Q(S)) \le (1 + C \delta^{\theta_0}) \HD^d(R). \label{eqn:ur_cor_11}
\end{align}

Let $R \in \Delta(R_0)$, and choose an index $i^* \in \N \cup \{0\}$ and a collection $S^*$ so that $R \in S^* \in \family_{i^*}$.
Assume first $R$ is the top cube of $S^*$, $R = Q(S^*)$.
Then by construction, 
\begin{align*}
\bigcup_{\substack{S \in \family_{i^*+1} \\ Q(S) \subsetneq R }} Q(S) = \bigcup_{\substack{Q \in m(S^*) \\ Q \subset R}} Q,
\end{align*}
where each of the unions above is a disjoint union.
In particular, we see
\begin{align*}
\sum_{\substack{S \in \family_{i^*+1} \\ Q(S) \subsetneq R }} \HD^d(Q(S)) & = \sum_{\substack{Q \in m(S^*) \\ Q \subset R}} \HD^d(Q) \\
& \le C \delta^{1- d\theta' - \theta} \HD^d(R) \\
& \le C \delta^{\theta_0} \HD^d(R),
\end{align*}
by (\ref{eqn:ur_cor_8}).
Now, for any index $k \in \N$, $k \ge i^* +1$, we have that
\begin{align*}
\bigcup_{\substack{ S \in \family_k \\ Q(S) \subsetneq R }} Q(S) & = \bigcup_{\substack{S \in \family_{k-1} \\ Q(S) \subsetneq R}} \left( \bigcup_{Q \in m(S)} Q \right)
\end{align*}
where each of the unions above are disjoint.
This gives
\begin{align}
\sum_{\substack{S \in \family_k  \\ Q(S) \subsetneq R }} \HD^d(Q(S)) & = \sum_{\substack{S \in \family_{k-1} \\ Q(S) \subsetneq R }} \left( \sum_{Q \in m(S)} \HD^d(Q) \right) \nonumber  \\
& \le C \delta^{\theta_0} \sum_{\substack{S \in \family_{k-1} \\ Q(S) \subsetneq R} } \HD^d(Q(S)) \nonumber \\
& \le (C \delta^{\theta_0})^{k - i^*} \HD^d(R), \label{eqn:ur_cor_9}
\end{align}
where the first inequality in the above follows from (\ref{eqn:ur_cor_8}), and the second is by induction on $k$.
As long as $\delta_0$ is small enough (depending only on $\theta_0$ and the underlying dimensions) , $C \delta_0^{\theta_0} < 1$.
Whence from (\ref{eqn:ur_cor_9}) we obtain
\begin{align*}
\sum_{\substack{S \in \family \\ Q(S) \subsetneq R}} \HD^d(Q(S)) & \le \sum_{k \ge i^* + 1} \sum_{\substack{S \in \family_{k} \\ Q(S) \subsetneq R}} \HD^d(Q(S)) \\
& \le \sum_{k \ge i^*+1} \left(C \delta^{\theta_0} \right)^{k - i^*} \HD^d(R) \\
& = \left( \sum_{k = 1}^\infty (C \delta^{\theta_0})^k \right) \HD^d(R)  \\
& \le C \delta^{\theta_0} \HD^d(R),
\end{align*}
since again, $\delta_0$ is small enough so that $\sum_{k \ge 1} (C \delta^{\theta_0})^k = 1 - (1 - C\delta^{\theta_0})^{-1}\le C \delta^{\theta_0}$.
We have thus shown that whenever $R \in \Delta(R_0)$ is a top cube, $R = Q(S^*)$, then 
\begin{align}
\sum_{\substack{S \in \family \\ Q(S) \subsetneq Q(S^*)}} \HD^d(Q(S)) \le C \delta^{\theta_0} \HD^d(Q(S^*)). \label{eqn:ur_cor_10}
\end{align}
From here, we deduce our estimate for general $R$.

Suppose that $R \in \Delta(R_0)$ is arbitrary.
Then we can decompose the collection of top cubes $Q(S) \subset R$, $S \in \family$ into 2 disjoint collections, $\mathcal{T}(R)$ and $\mathcal{R}(R)$, where
\begin{align*}
\mathcal{T}(R) & \coloneqq \{ S \in \family \; : \; Q(S) \subset R, \text{ and } Q(S) \subset Q(S') \subset R \text{ implies } S = S' \} \\
 \mathcal{R}(R) & \coloneqq \{ S \in \family\; : \; Q(S) \subset R, S \not \in \mathcal{T}(R) \}.
\end{align*}
Put simply, $\mathcal{T}(R)$ consists of the collections $S \in \family$ whose top cubes are the ``first'' descendants of $R$ that are top cubes, and $\mathcal{R}(R)$ are the rest.
Note that necessarily the cubes in $\mathcal{T}(R)$ are disjoint, and also to each $Q(S) \in \mathcal{R}(R)$ there is some $S' \in \mathcal{T}(R)$ for which $Q(S) \subsetneq  Q(S')$.
We estimate
\begin{align*}
\sum_{\substack{S \in \family \\ Q(S) \subset R}} \HD^d(Q(S)) & = \sum_{S \in \mathcal{T}(R)} \HD^d(Q(S)) + \sum_{S \in \mathcal{R}(R)} \HD^d(Q(S)) \\
& \le \sum_{S \in \mathcal{T}(R)} \HD^d(Q(S)) + \sum_{S' \in \mathcal{T}(R)} \left( \sum_{\substack{S \in \mathcal{R}(R) \\ Q(S) \subsetneq Q(S')}} \HD^d(Q(S)) \right)\\
& \le \sum_{S \in \mathcal{T}(R)} \HD^d(Q(S)) + \sum_{S' \in \mathcal{T}(R)} \left( \sum_{\substack{S \in \family \\ Q(S) \subsetneq Q(S')}} \HD^d(Q(S)) \right) \\
& \le \sum_{S \in \mathcal{T}(R)} \HD^d(Q(S))  + C \delta^{\theta_0} \sum_{S' \in \mathcal{T}(R)} \HD^d(Q(S')) \\
& = (1 + C \delta^{\theta_0}) \sum_{S \in \mathcal{T}(R)} \HD^d(Q(S)) \\
& \le (1 + C \delta^{\theta_0}) \HD^d(R).
\end{align*}
In the above, we used (\ref{eqn:ur_cor_10}) in the third inequality, and the fact that the cubes $Q(S)$ for $S \in \mathcal{T}(R)$ are disjoint and contained in $R$ in the last.
This shows (\ref{eqn:ur_cor_11}), and thus our proof is complete once we can justify how to construct the other $\family(R)$ as in Definition \ref{defn:corona} so that (\ref{cond:corona_topgr}) holds.

However, this last step is simple.
Notice that if $R \in \Delta_j$ where $\gen R = j = \gen R_0$, and also $R \subset (M/3) B_{R_0}$, then for $\delta_0$ sufficiently small, we have that $(M/3 \cd) B_R \subset M B_{R_0}$. Hence (\ref{eqn:ur_cor_1}) gives that
\begin{align*}
\HDE\left( \left( \dfrac{M}{3 \cd} B_R \right) \setminus \Gamma(S_0) \right) + \HD^d|_{\Gamma(S_0)} \left( \left( \dfrac{M}{3\cd} B_R \right) \setminus E \right) & \le C \delta^{1-d\theta'} (\diam R)^d,
\end{align*}
where $C >0$ depends only on $n, d$ and $C_D$.
In particular, we recall that this was the only condition we used on the Lipschitz graph $\Gamma(S_0)$ to be chosen for the top cube of $S_0$ in order to construct $\family$.
In particular, by simply taking $C >0$ larger in the conclusion of the Theorem, we can use this same Lipschitz graph $\Gamma(S_0)$ for each such $R$, and repeat the construction of $\family$ essentially verbatim to construction the partition $\family(R)$ of $\Delta(R)$ for each such $R$, finishing the proof of the Theorem. 
\end{proof}

\section{\texorpdfstring{$\delta$}{delta}-Corona decompositions imply \texorpdfstring{$\alpha_{\mu}(x,r)$}{alpha} are small}\label{sec:cor_alpha}

In this section, we show that \ref{cond:delta_cor} gives \ref{cond:alpha_c}, i.e., we prove Theorem \ref{thm:corona_alpha}.
Although the Carleson measure estimate we prove is a dyadic version of \ref{cond:alpha_c}, this discrete estimate appearing in Theorem \ref{thm:corona_alpha} implies the continuous one.
This estimate can be found in \cite[Lemma 5.9]{DFM19} so we omit the proof.
The first step in the proof is to obtain small-constant Carleson measure estimates on the $\alpha_\mu(x,r)$ when $\mu$ is a measure supported on a small-constant Lipschitz graphs with density close to $1$. 
This estimate is done in \cite{TOLSAALPHA} when $\mu$ is surface measure on the graph, but for completeness we fill in the gap when one takes $\mu$ slightly more general.

To state the Theorem in the same language as in \cite{TOLSAALPHA}, whenever $E = \spt \mu$ is $d$-Ahlfors regular, $\Delta$ is a system of dyadic cubes for $E$, and $Q \in E$, we abuse notation of $\alpha_\mu$ and set 
\begin{align}
\alpha_\mu(Q) \coloneqq \alpha_\mu(c_Q, 3 \, \diam Q) \label{eqn:alpha_cube}
\end{align}
where of course, $\alpha_\mu(c_Q, 3 \, \diam Q)$ is as in Definition \ref{defn:alpha}.
There is a very minor difference in $\alpha_\mu(Q)$ and that written in \cite{TOLSAALPHA}, where the normalization is taken with respect to the quantity $\ell(Q)$ (which is $\ell(Q) = 2^j$ when $Q \in \Delta_j$) in place of $3 \, \diam Q$, but since these quantities are comparable (with constant $C_D$) this difference is unimportant in the estimates that follow.
Let us state two main Theorems proved in \cite{TOLSAALPHA},  which we shall use in our proof of Theorem \ref{thm:corona_alpha}. 
First we recall the definition of ``large-constant'' uniformly rectifiable measures.
\begin{defn} \label{defn:ur}
Suppose that $\mu$ is a $d$-Ahlfors regular measure on $\R^n$. Then $\mu$ is said to be $d$-uniformly rectifiable (with constant $M >1$) if for each $x \in \spt \mu$ and $r >0$, there exists a Lipschitz map $f: B_d(0, r) \subset \R^d \ra \R^n$ with Lipschitz constant $\le M$, so that 
\begin{align*}
\mu(B(x,r) \cap f(B_d(0,r))) \ge M^{-1} r^d.
\end{align*}
\end{defn}

\begin{thm}[Theorem 1.2 in \cite{TOLSAALPHA}]\label{thm:tolsa_ur_alpha}
Let $\mu$ be a $d$-Ahlfors regular measure in $\R^n$ with constant $C_\mu$, and suppose that $\mu$ is $d$-uniformly rectifiable (with large constant, $M > 1$). Fix a system $\Delta$ of dyadic cubes for $\mu$ with constant $\cd > 0$. 
Then there is some constant $C_0 = C_0(n,d, C_\mu, M, \cd) >1$ so that one has the Carleson condition 
\begin{align}
\sup_{R_0 \in \Delta} \mu(R_0)^{-1} \sum_{Q \in \Delta(R_0)} \alpha_\mu(Q)^2 \mu(Q) \le C_0. \label{eqn:tolsa_ur_alpha}
\end{align}
\end{thm}

\begin{thm}[Theorem 1.1 and Remark 4.1 that follows in \cite{TOLSAALPHA}] \label{thm:tolsa_lip}
Suppose that $\Gamma$ is a $d$-dimensional Lipschitz graph in $\R^n$ with constant $\delta < 1$, let $\mu = \HD^d|_\Gamma$, and suppose that $\Delta$ is a system of dyadic-cubes for $E$ with constant $\cd > 0$.
Then the following Carleson condition holds:
\begin{align*}
\sup_{R_0 \in \Delta} \mu(R_0)^{-1} \sum_{Q \in \Delta(R_0)} \alpha_\mu(Q)^2 \mu(Q) \le C_0 \delta^2,
\end{align*}
where $C_0 = C_0(n,d, \cd)$ is independent of $\delta$. 
\end{thm}

To be totally transparent, Remark 4.1 in \cite{TOLSAALPHA} is stated for true dyadic cubes in $\R^n$ that meet the Lipschitz graph $\Gamma$, and the result is stated as a global Carleson packing condition for compactly supported Lipschitz graphs.
However, it is straight-forward to deduce the Theorem above from how Remark 4.1 is stated.
Indeed, one can obtain a local result from a global one by fixing an initial cube $R_0$ of $\Gamma$, and finding a $(C\delta)$-Lipschitz graph $\Gamma'$ that agrees with $\Gamma$ in $10 B_{R_0}$ and has support in $20 B_{R_0}$.
Then Theorem \ref{thm:tolsa_lip} applied to $\Gamma'$ gives the result, since the $\alpha_\mu(Q)$ are local to $3B_{R_0}$ anyway. 

From here, we can extend this small-constant estimate to measures of the form $d\mu(x) = g(x) \; d\HD^d|_\Gamma(x)$ where $g(x)$ is some controlled density, using the following two Lemmas.

\begin{lemma}\label{lemma:osc}
Suppose that $\mu$ is a $d$-Ahlfors regular in $\R^n$ with constant $C_\mu > 0$, and $\Delta$ is a system of dyadic cubes for $\spt \mu$ with constant $\cd >0$.
Then for any $g \in L^2_{\mathrm{loc}}(\mu)$, the coefficients $\osc_{\mu, g}(Q)$ defined on the cubes $Q$ by
\begin{align} \label{eqn:oscdef}
\osc_{\mu, g}(Q) \coloneqq ( \diam Q) ^{-d-1} \inf_{\lambda \in \R } \sup_{f \in \Lambda(3 B_Q )} \abs{\int   fg - \lambda f \; d\mu }
\end{align}
satisfy the Carleson condition
\begin{align*}
\sup_{R_0 \in \Delta} \mu(R_0)^{-1} \sum_{Q \in \Delta(R_0)} \osc_{\mu, g}(Q)^2 \mu(Q) \le C_0 \left(  \fint_{C_0 R_0} \abs{g - (g)_{C_0 R_0}}^2 \; d\mu  \right).
\end{align*}
Here, $C_0 = C_0(n, d, \cmu, \cd) > 0 $ is independent of $g$. 
\end{lemma}
\begin{proof}
The idea is to use a dyadic Martingale decomposition of $g$ with respect to the dyadic cubes for $E = \spt \mu$, but for a family of adjacent dyadic cubes for $E$ rather than the single system from Section \ref{subsec:lattice} (we shall see the flexibility this gives us shortly).
Lemma 2.2 in \cite{AD20} (which uses Theorems 2.9 and 5.9 in \cite{HT14}) imply that there exists $\delta \in (0,1)$ small, and $M \in \N$, $C_{\mathcal{D}} > 1$, large depending only on $n, d$ and $C_\mu$, so that the following holds. For each $\omega \in \{1, \dotsc, M\}$, one can find a system of $\mathcal{D}(\omega) = \cup_{j \in \Z} \mathcal{D}_j(\omega)$ of ``dyadic cubes'' for $E$ such that 
\begingroup
\addtolength{\jot}{0.5em}
\begin{align}
& \text{ \parbox[t]{0.8 \textwidth}{for all $j \in \Z$, $E = \cup_{Q \in \mathcal{D}_j(\omega)} Q,$ } }  \\
& \text{ \parbox[t]{0.8 \textwidth}{if $R, R' \in \mathcal{D}_j(\omega)$ and $R \ne R'$, then $\HD^d(R \cap R') = 0$, } }  \\
& \text{ \parbox[t]{0.8 \textwidth}{for each $j \le \ell$ and $Q \in \mathcal{D}_\ell(\omega)$, one has $Q = \cup_{R \subset Q, R \in \mathcal{D}_j(\omega)}R$, } }  \\
& \text{ \parbox[t]{0.8 \textwidth}{each $Q \in \mathcal{D}_j(\omega)$ has a ``center,'' $z_Q$ so that $B(z_Q, \delta^j/5) \cap E \subset Q \subset B(z_Q, 3 \delta^j)$. Consequently, $\HD^d(Q) \simeq_{\cmu} \delta^{jd}$.} }  
\end{align}
\endgroup
These first properties are essentially the same as the system $\Delta$ from Section \ref{subsec:lattice}, but here the $\{ \mathcal{D}(\omega) \; : \; 1 \le \omega \le M\}$ also satisfy 
\begin{align}
& \text{ \parbox[t]{0.8 \textwidth}{for any $x \in E, r >0$, there is some choice of $1 \le \omega \le M$, $j \in \Z$ with $C_{\mathcal{D}}^{-1} r \le \delta^j \le C_{\mathcal{D}} r$, and $Q \in \mathcal{D}_j(\omega)$ so that $B(x,r) \subset Q$.} }  \label{cond:dyadicp5}
\end{align}
We shall use the family of dyadic systems $\{ \mathcal{D}(\omega)\}_{\omega=1}^{M}$ in conjunction with the fixed dyadic system $\Delta$ from Section \ref{subsec:lattice} to prove the result.

Let $Q_0 \in \Delta$ be a dyadic cube of $E$ from Section \ref{subsec:lattice}, and fix any cube $R \subset Q_0$.
By \eqref{cond:dyadicp5}, there exists some $1 \le \omega_R \le M$, $j_R \in \Z$ with $\delta^{j_R} \simeq_{C_{\mathcal{D}}} \diam R$, and $Q_R \in \mathcal{D}_{j_R}(\omega_R)$ so that $3B_R \subset Q_R$.
Now since $g \chi_{Q_R} \in L^2(\mu)$, we may write
\begin{align}
(g - (g)_{Q_R})  \chi_{Q_R} & = \sum_{S \in \mathcal{D}(\omega_R, Q_R)} \Delta_{S} g  \label{eqn:oscred1}
\end{align}
with convergence in $L^2(\mu)$, where by definition, $\mathcal{D}(\omega_R, Q_R)$ are the subcubes of $Q_R$ (in $\mathcal{D}(\omega_R)$), \[\Delta_S g = \sum_{S' \in \mathrm{child}(S)} ((g)_{S'} - (g)_S) \chi_{S'},\] and $(g)_A = \fint_A g \; d\mu$. Moreover, one has that \[ \|(g - (g)_{Q_R}) \chi_{Q_R} \|_{L^2(\mu)}^2 = \sum_{S \in \mathcal{D}(\omega_R, Q_R)} \| \Delta_S g \|_{L^2(\mu)}^2, \] since the terms on the right-hand side of \eqref{eqn:oscred1} are pairwise orthogonal in $L^2(\mu)$.

Now since $\osc_{\mu, g}(R) = \osc_{\mu, g + a} (R)$ for any constant $a$, we may assume that $(g)_{Q_R} = 0$ in our estimate of $\osc_{\mu, g}(R)$. Fix $f \in \Lambda(3B_R)$, and note that since $f$ is $1$-Lipschitz and $\spt f \subset Q_R$,
\begin{align*}
\left | \int gf \; d\mu \right| & \le \sum_{S \in \mathcal{D}(\omega_R, Q_R)} \left| \int \Delta_S g f \; d \mu \right | \\
& = \sum_{S \in \mathcal{D}(\omega_R, Q_R)} \left| \int \Delta_S g (f - f(z_S)) \; d \mu \right | \\
& \le \cmu \sum_{S \in \mathcal{D}(\omega_R, Q_R)} (\diam S)^{1+d/2} \| \Delta_S g \|_{L^2(\mu)}.
\end{align*}
Since $f$ was arbitrary, taking $\lambda = 0$ in (\ref{eqn:oscdef}) and recalling that $\mu(R) \simeq_{\cmu} (\diam R)^d$ we obtain for some constant $C$ depending only on $n, d$ and $\cmu$,
\begin{align*}
\osc_{\mu, g}(R)^2 \mu(R) & \le C \left( \sum_{S \in \mathcal{D}(\omega_R, Q_R)} \dfrac{(\diam S)^{1 + d/2} }{(\diam R)}  ||\Delta_S g||_{L^2(\mu)}\right)^2 \dfrac{1}{\mu(R)} .
\end{align*}
Using the fact that $\diam Q_R \simeq_{C_{\mathcal{D}}} \diam R$, the Cauchy-Schwarz inequality, and the estimate $\left( \sum_{S \in \mathcal{D}(\omega_R, Q_R)} \left(\dfrac{\diam S}{ \diam R} \right) \mu(S) \right) \le C \mu(Q_R) \le C \mu(R)$ give
\begin{align}
\sum_{R \in \Delta(Q_0)}  \osc_{\mu, g}(R)^2 \mu(R) & \le C \sum_{R \in \Delta(Q_0)}   \left( \sum_{S \in \mathcal{D}(\omega_R, Q_R)} \left(\dfrac{\diam S}{ \diam R}\right) ||\Delta_S g||_{L^2(\mu)}^2 \right) \nonumber \\
& \times  \left( \sum_{S \in \mathcal{D}(\omega_R, Q_R)} \left(\dfrac{\diam S}{ \diam R} \right) \mu(S) \right)  \dfrac{1}{\mu(R)} \nonumber \\
& \le C \sum_{R \in \Delta(Q_0)}  \left( \sum_{S \in \mathcal{D}(\omega_R, Q_R)} \left(\dfrac{\diam S}{ \diam Q_R}\right) ||\Delta_S g||_{L^2(\mu)}^2 \right).  \label{eqn:osc1}
\end{align}
Setting $\mathcal{I}(\omega) = \{ Q \in \mathcal{D}(\omega) \; : \; Q \subset C_2 Q_0 \}$ for some large $C_2 >1$ depending only on $n, d, \cmu, \cd$ and $C_{\mathcal{D}}$, we can crudely estimate \eqref{eqn:osc1} by
\begin{align}
C \sum_{\omega=1}^M \sum\limits_{Q \in \mathcal{I}(\omega)} \sum_{S \in \mathcal{D}(\omega, Q)} \left(\dfrac{\diam S}{ \diam Q}\right) ||\Delta_S g||_{L^2(\mu)}^2 \label{eqn:osc2}
\end{align}
for some large constant $C>1$ depending on the same parameters. This is because each term in the summand of \eqref{eqn:osc1} appears in \eqref{eqn:osc2}, and to each term in \eqref{eqn:osc2} there are only finitely many terms coming from \eqref{eqn:osc1} which are associated to those of \eqref{eqn:osc2}, by Ahlfors regularity and the fact that $\diam Q_R \simeq_{C_{\mathcal{D}}} R$. 

Finally, switching the order of summation in \eqref{eqn:osc2} and using the fact that  \[\sum_{Q \in \mathcal{I}(\omega), Q \supset S} \left( \dfrac{\diam S}{\diam Q} \right) \le C,\] we see that
\begin{align*}
\sum_{R \in \Delta(Q_0)}\osc_{\mu, g}(R)^2 \mu(R) &  \le C  \sum_{\omega=1}^M \sum_{Q \in \mathrm{top}(\mathcal{I}(\omega))} \sum_{S \in \mathcal{D}(\omega, Q)} || \Delta_S g||_{L^2(\mu)}^2 \\
& \le  C \sum_{\omega=1}^M \sum_{Q \in \mathrm{top}(\mathcal{I}(\omega))}  \| (g - (g)_Q) \chi_Q    \|_{L^2(\mu)}^2 \\
& \le C \mu(Q_0) \left(  \fint_{C_2 Q_0} \abs{ g - (g)_{C_2 Q_0}}^2 \; d\mu  \right) 
\end{align*}
where $\mathrm{top}(\mathcal{I}(\omega))$ are the maximal cubes in $\mathcal{I}(\omega)$ whose boundaries intersect in sets of zero $\HD^d$ measure. This completes the proof, since the constants $\delta, M$ and $C_{\mathcal{D}}$ depend only on $n,d$ and $\cmu$.
\end{proof}

\begin{lemma} \label{lemma:alpha_density}
Suppose that $\mu$ is a $d$-Ahlfors regular measure with constant $(1 + \delta) < 2$.
Moreover, assume that $E = \spt \mu$ is $d$-rectifiable.
Fix $\Delta$, a system of dyadic cubes for $E$ with constant $\cd > 0$.
Then for the function $g = d\mu/d\HD^d|_E$, we have
\begin{align*}
C_0^{-1} \left( - C_0 \osc_{\mu, g^{-1}}(Q) + \alpha_{\HD^d|_E}(Q) \right) \le \alpha_{\mu}(Q) \le C_0 \left ( \osc_{\HD^d|_E, g}(Q) + \alpha_{\HD^d|_E}(Q) \right),
\end{align*}
where $C_0 = C_0(n,d, \cd) > 0$ is independent of $\mu$ and $\delta$.
\end{lemma}

\begin{proof}
By virtue of Lemma \ref{lemma:adr_m_adr_set}, we know that the densities $d \HD^d|_E/d\mu$ and $d\mu/d\HD^d|_E$ exist and satisfy 
\begin{align*}
(1+ \delta)^{-1} \le \dfrac{d\HD^d|_E}{d\mu}, \dfrac{d\mu}{d\HD^d|_E} \le (1+\delta)
\end{align*}
$\mu$-almost everywhere.
Set $g \coloneqq d\mu/d\HD^d|_E$.
We readily compute for any dyadic cube $Q \in \Delta$, any $f \in \Lambda(3 B_Q)$, $\lambda \in \R$ and $\nu$ a $d$-flat measure,
\begin{align*}
\abs{ \int f ( d \mu -d \nu) } &  = \abs{ \int f \left( g d \HD^d|_E -d \nu \right) }  \\
& \le \abs{ \int f \left (  g - \lambda \right ) d\HD^d|_E } + \abs{ \int f \left( \lambda d\HD^d|_E - d\nu \right ) }.
\end{align*}
Taking the infimum over $\lambda \in \R$ which minimizes the quantity $\osc_{\HD^d|_E, g}$ (which one readily sees must be in $[(1+\delta)^{-1}, (1+ \delta)]$, by the bounds on $g$), and then taking the infimum over flat measures in the definition of $\alpha_{\HD^d|_E}$, we obtain
\begin{align*}
\alpha_{\mu}(Q) \le C \left( \osc_{\HD^d|_E,g}(Q)  + \alpha_{\HD^d|_\Gamma}(Q) \right),
\end{align*}
with constant $C >0$ depending on $n, d$ and $\cd$.
Reversing the roles of $\mu$ and $\HD^d|_E$ gives the other inequality.
\end{proof}

Putting together the previous Lemmas and Theorem \ref{thm:tolsa_lip}, we obtain our small-constant Carleson measure estimate we shall use in proving Theorem \ref{thm:corona_alpha}.
\begin{cor} \label{cor:alpha_lip}
Suppose that $\Gamma$ is a $d$-dimensional $\delta$-Lipschitz graph in $\R^n$, and $\mu$ is a $d$-Ahlfors regular measure with support $\Gamma$ and with constant $(1 + \delta)$.
Fix $\Delta$, a system of dyadic cubes for $\Gamma$ with constant $\cd$.
Then one has the Carleson packing condition,
\begin{align*}
\sup_{R_0 \in \Delta} \mu(R_0)^{-1} \sum_{Q \in \Delta(R_0)} \alpha_\mu(Q)^2 \mu(Q) \le C_0 \delta^2.
\end{align*}
Here $C_0 = C_0(n,d, \cd) > 0$ is independent of $\mu$ and $\delta$. 
\end{cor}
\begin{proof}
Set $g \coloneqq d\mu/d\HD^d|_\Gamma$.
Lemma \ref{lemma:adr_m_adr_set} implies that $g, g^{-1} \le (1 + \delta)$.
We combine the cube-wise inequality from Lemma \ref{lemma:alpha_density} along with the Carleson packing conditions from Theorem \ref{thm:tolsa_lip} and Lemma \ref{lemma:osc} to see that for any $R_0 \in \Delta$ and $\osct g = \esssup g - \essinf g$, 
\begin{align*}
\mu(R_0)^{-1} \sum_{Q \in \Delta(R_0)} \alpha_\mu(Q)^2 \mu(Q) & \le C \mu(R_0)^{-1} \left( \sum_{Q \in \Delta(R_0)} \left( \osc_{\HD^d|_\Gamma, g}(Q)^2 + \alpha_{\HD^d|_\Gamma}(Q)^2 \right) \mu(Q) \right) \\
& \le C (\left( \osct{g}\right)^2 + \delta^2) \\
& \le  C \left( \left( (1 + \delta) - (1+\delta)^{-1} \right)^2 + \delta^2 \right) \le  C \delta^2,
\end{align*}
completing the proof of the Corollary.
\end{proof}

Finally, we transfer the $\alpha$-Carleson packing conditions from Lipschitz graphs to $\delta$-UR sets with the small-constant Corona decomposition, i.e., we prove \ref{cond:delta_cor} implies \ref{cond:alpha_c}. The upper Ahlfors regular in this implication essentially comes for free.

\begin{thm} \label{thm:corona_alpha}
There are constants $C_0 > 1$ and $\delta_0, \theta_0 \in (0,1)$ depending only on the dimensions $n$ and $d$ so that the following holds.
Whenever $0 < \delta < \delta_0$, $E \subset \R^n$ is $d$-Ahlfors regular and admits $\delta$-Corona decompositions, then for any measure $\mu$ of the form $d\mu(x) = g(x) d\HDE(x)$ where $g$ is Borel and satisfies $(1 + \delta)^{-1} \le g \le (1 + \delta)$, we have the Carleson condition 
\begin{align}
\sup_{R_0 \in \Delta} \mu(R_0)^{-1} \sum_{Q \in \Delta(R_0)} \alpha_\mu(Q)^2 \mu(Q) \le C_0 \delta^{\theta_0}.
\end{align}
Moreover, $E$ is upper $d$-Ahlfors regular with constant $1+C_0 \delta^{\theta_0}$.

Here $\Delta$ is a fixed system of dyadic cubes for $E$ with bounded constant $\cd \le C_{n,d}$ coming from the definition of $\delta$-Corona decompositions. 
\end{thm}

\begin{proof}
Choose a system of dyadic cubes $\Delta$ for $E$ as in the definition of a $\delta$-Corona decomposition with constant $\cd \le C_{n,d}$. We begin the proof with a reduction.
Recall that by Lemma \ref{lem:cor_ur}, $E$ is $C \delta$-UR as long as $\delta_0$ is sufficiently small.
By Lemma \ref{lemma:delta_ur_adr} then, we may assume that $E$ is Ahlfors regular with constant $(1+ C_0 \delta^{1/d}) \le (1 + C_0 \delta_0^{1/d})$.
Let $\mubar = \HD^d|_E$.
Whenever $R_0 \in \Delta$, denote by $\family(R_0)$ the partition of $\Delta(R_0)$ given by the $\delta$-Corona decomposition of $E$ in Definition \ref{defn:corona}.
Moreover, denote by $S_{R_0} \in \family(R_0)$ the subcollection of $\Delta(R_0)$ containing $R_0$.
We show that the conclusion of the Theorem follows from the estimate
\begin{align}
\sup_{R_0 \in \Delta} \mubar(R_0)^{-1} \sum_{Q \in S_{R_0}} \alpha_{\mubar}(Q)^2 \mubar(Q) \le C_0 \delta^{\theta_0}. \label{eqn:cor_alpha_1}
\end{align}

Indeed, assume that (\ref{eqn:cor_alpha_1}) holds, and let $R_0 \in \Delta$ be given.
Choose $\family(R_0)$ as in the Definition of a $\delta$-Corona decomposition.
Then
\begin{align*}
\mubar(R_0)^{-1} \sum_{Q \in \Delta(R_0)} \alpha_\mubar(Q)^2 \mubar(Q) & = \mubar(R_0)^{-1} \sum_{Q \in S_{R_0}} \alpha_\mubar(Q)^2 \mubar(Q) \\
& \qquad + \mubar(R_0)^{-1} \sum_{S \in \family(R_0) \setminus \{S_{R_0} \}} \sum_{Q \in S} \alpha_\mubar(Q)^2 \mubar(Q) \\
& \stackrel{(\ref{eqn:tolsa_ur_alpha})}{\le} C_0 \delta^{\theta_0} + C \mubar(R_0)^{-1} \sum_{S \in \family(R_0) \setminus \{ S_{R_0} \} } \mubar(Q(S)) \\
& \stackrel{(\ref{cond:corona_carleson_cor})}{\le} C_0 \delta^{\theta_0} + C \delta,
\end{align*}
as long as $\delta_0 < 1$, where $C >0$ is some dimensional constant depending only on $n, d$ here and in the future.
Next, one argues just as in the proof of Corollary \ref{cor:alpha_lip} to replace $\mubar$ with $\mu$ to obtain
\begin{align*}
\mu(R_0)^{-1} \sum_{Q \in \Delta(R_0)} \alpha_\mu(Q)^2 \mu(Q) & \le C \left( C_0 \delta^{\theta_0}  +\left( \osct \dfrac{d\mu}{d\mubar} \right)^2 \right) \\
& \le C C_0 \delta^{\theta_0} + C \delta^2.
\end{align*}
This shows that it suffices to demonstrate (\ref{eqn:cor_alpha_1}), and thus from here on we fix $R_0 \in \Delta$, $\family = \family(R_0)$, and we may as well assume in our estimates that $\mu = \mubar = \HD^d|_E$.
Fix the top family $S \equiv S_{R_0} \in \family$ to be the subcollection for which $R_0 \in S_{R_0}$, and choose an approximating $\delta$-Lipschitz graph $\Gamma = \Gamma(S)$ as in (\ref{cond:corona_graph}) that also satisfies (\ref{cond:corona_topgr}).
Notice also that $R_0 = Q(S)$.
We break the remainder of the proof into three steps.

\medskip
\noindent \textbf{Step one}: We define a measure on $\Gamma \cap 3 B_{Q(S)}$ to compare to $\mu$.
Recall that $c_{Q(S)}$ is the center of $Q(S)$, and that $B_{Q(S)} = B(c_{Q(S)}, \diam Q(S)) \supset Q(S)$.
The condition (\ref{cond:corona_graph}) on the proximity of $\Gamma$ to $E$ near $\delta^{-1} B_{Q(S)}$ implies that there is $c_{\Gamma, S} \in \Gamma \cap B_{Q(S)}$ so that $\abs{c_{Q(S)} - c_{\Gamma, S}} \le \delta \, \diam Q(S)$. 

Now we produce a system of dyadic cubes for $\Gamma$ that come from true dyadic cubes in $\R^n$ as follows.
Note that up to rotation, we may assume that $\Gamma$ is the graph of a Lipschitz function over $\R^d \subset \R^n$.
Let $Q^\Gamma(S)$ be a (true) closed cube in $\R^n$ with axis-parallel sides centered at $c_{\Gamma, S}$ such that $10 B_{Q(S)} \subset Q^\Gamma(S) \subset 20 \sqrt{n} B_{Q(S)}$, so that $\diam Q^\Gamma(S) \simeq_{n,d} \diam Q(S)$.
Denote by $\tilde{Q}^\Gamma(S)$ the projection of $Q^\Gamma(S)$ onto $\R^d$, and notice that $\tilde{Q}^\Gamma(S)$ is a true cube in $\R^d$, since $Q^\Gamma(S)$ has axis-parallel sides.
Split $\tilde{Q}^\Gamma(S)$ into $2^d$ closed subcubes $\tilde{Q}^\Gamma_{1,1} \dotsc, \tilde{Q}^\Gamma_{1, 2^d}$ of $\tilde{Q}^\Gamma(S) \subset \R^d$ with disjoint interiors.
Denote this collection of cubes in $\R^d$ by $\Delta_1^\Gamma(\tilde{Q}^\Gamma(S))$, which we call first generation (true) dyadic cubes in $\R^d$ contained in $\tilde{Q}^\Gamma(S)$.
Then one generates the family $\Delta^\Gamma_j(\tilde{Q}^\Gamma(S))$ from $\Delta^\Gamma_{j-1}(\tilde{Q}^\Gamma(S))$ by splitting each cube in the previous generation into $2^d$ more (true) closed cubes in $\R^d$.
With $\Delta_0^\Gamma(\tilde{Q}^\Gamma(S)) = \{ \tilde{Q}^\Gamma(S)\}$, denote by $$ \Delta^\Gamma(\tilde{Q}^\Gamma(S)) = \bigcup_{j \ge 0} \Delta_j^\Gamma(\tilde{Q}^\Gamma(S)).$$

We then lift the dyadic cubes in $\Delta^\Gamma(\tilde{Q}^\Gamma(S))$ to closed cubes in $\R^n$ centered on $\Gamma$.
That is, if $\tilde{Q}^\Gamma \in \Delta^\Gamma_j(\tilde{Q}^\Gamma(S))$ with center $c_{\tilde{Q}^\Gamma}$, then since $\Gamma$ is a $\delta$-Lipschitz graph defined over $\R^d$, there is a unique $c_{Q^\Gamma} \in \Gamma$ so that $\pi_{\R^d}(c_{Q^\Gamma}) = c_{\tilde{Q}^\Gamma}$.
Let $Q^\Gamma$ be a closed cube in $\R^n$ with axis-parallel sides, center equal to $c_{Q^\Gamma}$, and side-length equal to that of $\tilde{Q}^\Gamma$.
Denote the collection of all such cubes generated this way by $\Delta_j^\Gamma(Q^\Gamma(S))$, and set \[\Delta^\Gamma \equiv \Delta^\Gamma(Q^\Gamma(S)) = \bigcup_{j \ge0} \Delta_j^\Gamma(Q^\Gamma(S)).\] Notice by construction we have the following facts about the dyadic cubes in $\Delta^\Gamma(Q^\Gamma(S))$:
\begingroup
\addtolength{\jot}{0.5em}
\begin{align}
& \text{ \parbox[t]{0.8 \textwidth}{ for each $j \ge 0$, the cubes in $\Delta_j^\Gamma(Q^\Gamma(S))$ have disjoint interiors. Moreover, $\Gamma \cap Q^\Gamma(S) = \cup_{Q^\Gamma \in \Delta_j^\Gamma(Q^\Gamma(S))} ( \Gamma \cap \text{Int}(Q^\Gamma) )  \cup F_j$, where $F_j \subset \Gamma$ is an $\HD^d$-null set. }} \label{cond:sgam_1} \\
& \text{ \parbox[t]{0.8 \textwidth}{ if $Q^\Gamma \in \Delta_j^\Gamma(Q^\Gamma(S))$ where $j \in \N$,  then there is a unique $R^\Gamma  \in \Delta_{j-1}^\Gamma(Q^\Gamma(S))$ so that $Q^\Gamma \subset R^\Gamma $, } } \label{cond:sgam_2} \\
& \text{ \parbox[t]{0.8 \textwidth}{ for each $Q^\Gamma \in \Delta^\Gamma$, one has that $$ (1 + C \delta)^{-1} \HD^d(\pi_{\R^d}(Q^\Gamma)) \le  \HD^d(\Gamma \cap Q^\Gamma) \le (1 + C \delta) \HD^d(\pi_{\R^d}(Q^\Gamma)) .$$  } } \label{cond:sgam_3}
\end{align}
\endgroup
Indeed, conditions \eqref{cond:sgam_1}, (\ref{cond:sgam_2}), and (\ref{cond:sgam_3}) follow from the fact that $\Gamma$ is a $\delta$-Lipschitz graph as long as $\delta_0$ is chosen small enough.

Fix parameters $M, \lambda >1$ to be determined.
In the end, the choice of these parameters shall depend only on the underlying dimensions $n$ and $d$.
Since $E$ is $d$-Ahlfors, the set 
\begin{align*}
\left \{ R \in \Delta_{\gen R_0} \; : \; R \cap 20 \sqrt{n} B_{Q(S)} \ne \emptyset \right \}
\end{align*}
has boundedly many elements, $R_1, R_2, \dotsc, R_k$ with $k \le C$, again for some dimensional constant $C$ depending only on $n$ and $d$.
As per (\ref{cond:corona_topgr}), we know that for each $1 \le j \le k$, there is a partition $\family_i = \family(R_i)$ of $\Delta(R_i)$ that satisfy the conditions of the $\delta$-Corona decomposition, (\ref{cond:corona_coherent})-(\ref{cond:corona_carleson}).
Moreover, such a partition can be chosen so that when $S_i \in \family_i$ is the subcollection with $R_i \in S_i$, we know that the conditions (\ref{cond:corona_coherent})-(\ref{cond:corona_carleson}) are satisfied with the same Lipschitz graph $\Gamma$ as the one chosen for $S$. 

We perform a stopping time argument on the cubes in $\Delta^\Gamma$ to find some coherent collection $S^\Gamma \subset \Delta^\Gamma$ of (true) cubes in $\R^n$ for which $\Gamma$ and $E$ are sufficiently close.
Let us say that a cube $Q^\Gamma \in \Delta^\Gamma$ is `far from $S$' (written FS) if the cube $\lambda Q^\Gamma$ does not meet any $Q \in S \cup S_1 \cup \cdots \cup S_k \eqqcolon S^*$ satisfying 
\begin{align}
M^{-1} \diam Q^\Gamma \le \diam Q \le M \diam Q^\Gamma. \label{eqn:cor_alpha_st1}
\end{align}
Now we proceed generation by generation.
If $M, \lambda$ are chosen large enough and if $\delta_0$ is sufficiently small (depending on only on the underlying dimensions), then one readily checks that the first few generations of cubes in $\Delta^\Gamma$ are not FS.
Therefore, let the first generation of $\Delta^\Gamma$ be in the set $S^\Gamma$. We continue generation by generation in the dyadic system $\Delta^\Gamma$ adding cubes to $S^\Gamma$, until we reach a cube $Q^\Gamma$ which has a sibling (possibly itself) which is FS.
At this stage, $Q^\Gamma$ and all of its siblings become minimal cubes of $S^\Gamma$, denoted $m(S^\Gamma)$, and no other subcubes from the parent of $Q^\Gamma$ are added to $S^\Gamma$. 
Notice that this gives that the collection $S^\Gamma$ is coherent, and its minimal cubes are disjoint and contained in $Q^\Gamma(S)$.
In addition, if $Q^\Gamma \in m(S^\Gamma)$, then the fact that its parent, $R^\Gamma$ is not FS implies that $\lambda R^\Gamma$ meets some element $Q \in S^*$ with 
\begin{align*}
M^{-1} \diam R^\Gamma \le \diam Q \le M \diam R^\Gamma,
\end{align*}
As long as $\delta_0$ is taken sufficiently small (depending on $M$ and $\lambda$), then the fact that $\diam R^\Gamma$ and $\diam Q$ are comparable and the fact that $\lambda R^\Gamma$ meets $Q$ implies that 
\begin{align*}
\delta^{-1} B_Q \supset \delta_0^{-1} B_Q \supset R^\Gamma.
\end{align*}
Then condition (\ref{cond:corona_graph}) implies that 
\begin{align}
\sup_{x \in (E \cup \Gamma) \cap R^\Gamma} \dist(x,E) + \dist(x, \Gamma) & \le \delta \diam Q   \le \delta M \diam R^\Gamma. \label{eqn:cor_alpha_st2}
\end{align}
Since $Q^\Gamma \subset R^\Gamma$, this implies
\begin{align}
\sup_{x \in (E \cup \Gamma) \cap Q^\Gamma} \dist(x,E) + \dist(x, \Gamma) \le  2 \delta M \diam Q^\Gamma. \label{eqn:cor_alpha_st2p}
\end{align}

Now, we claim that for any cube $Q^\Gamma$ satisfying (\ref{eqn:cor_alpha_st2p}) (and thus, for all cubes $Q^\Gamma \in S^\Gamma$), we have that 
\begin{align}
1 - C_1\delta^{\theta_1} \le \dfrac{\HD^d|_E(Q^\Gamma)}{\HD^d|_\Gamma(Q^\Gamma)} \le 1 + C_1 \delta^{1/d} \label{eqn:cor_alpha_3}
\end{align}
for constants $C_1 >1$ and $\theta_1 \in (0,1/d)$ depending only the underlying dimensions $n$, $d$, and the constant $M$ as long as $\delta_0$ is small.
The proof of this fact is a bit tedious, but the main ideas are that $\Gamma$ is the graph of a $\delta$-Lipschitz graph defined over $\R^d$, and there is another $\delta$-Lipschitz graph $\Gamma'$ that gives a good (measure) approximation to $E$ inside $Q^\Gamma$ in the sense that 
\begin{align}\label{eqn:cor_alpha_3p}
\HD^d( Q^\Gamma \cap (E \Delta \Gamma')) \le C \delta (\diam Q^\Gamma)^d,
\end{align}
by virtue of (\ref{cond:corona_graph}).
Then the fact that $Q^\Gamma$ is centered on $\Gamma$, and estimate (\ref{eqn:cor_alpha_st2p}) implies that $\Gamma'$ also passes near the center of $Q^\Gamma$ in that  $\dist(c_{Q^\Gamma}, \Gamma') \le C( \delta^{1/d} +  \delta M ) \diam Q^\Gamma \le C \delta^{1/d} \diam Q^\Gamma$. Moreover, $\Gamma'$ can be written as a $C \delta^{1/d}$-Lipschitz graph over $\R^d$, so that since $Q^\Gamma$ is a true cube with sides parallel to the coordinate axes, we can estimate with the area formula that
\begin{align*}
(\diam  Q^\Gamma)^d \le \HD^d|_\Gamma(Q^\Gamma), \;  \HD^d|_{\Gamma'}(Q^\Gamma) \le (1 + C \delta^{1/d}) (\diam Q^\Gamma)^d.
\end{align*}
Along with \eqref{eqn:cor_alpha_3p}, the estimate above on $\HD^d|_{\Gamma'}(Q^\Gamma)$ then implies 
\begin{align*}
(1 - C \delta) (\diam Q^\Gamma)^d \le \HD^d|_E(Q^\Gamma) \le (1 + C \delta^{1/d}) (\diam Q^\Gamma)^d,
\end{align*}
so that the estimates on $\HD^d|_\Gamma(Q^\Gamma)$ in \eqref{eqn:cor_alpha_3p} then give \eqref{eqn:cor_alpha_3} with any $\theta_1 \in (0, 1/d)$ and $C_1$ depending on $\theta_1$.

Finally, we define our density for our measure on $\Gamma$.
Define the coefficients $b_{Q^\Gamma} \coloneqq \HD^d|_E(Q^\Gamma)/\HD^d|_\Gamma(Q^\Gamma)$, set
\begin{align*}
g(x) & \coloneqq \begin{cases}
b_{Q^\Gamma}, & \text{when } x \in \Gamma \cap Q^\Gamma, \text{ and } Q^\Gamma \in m(S^\Gamma) \\
1 & \text{ otherwise, }
\end{cases}
\end{align*}
and define $d\gamma(x) \coloneqq g(x) d\HD^d|_\Gamma(x)$.
Notice that by (\ref{eqn:cor_alpha_3}), we have that $1 - C_1(M \delta)^{\theta_1} \le g(x) \le 1 + C_1(M \delta)^{\theta_1}$ on $\Gamma$, so that since $\Gamma$ is a $\delta$-Lipschitz graph, we have that 
\begin{align}
& \text{ \parbox[t]{0.8 \textwidth}{ $\gamma$ is a $d$-Ahlfors regular measure with support $\Gamma$ and constant $1 + C_1 (M \delta)^{\theta_1}$. } } \label{eqn:cor_alpha_4}
\end{align}
for some (larger) $C_1 >1$, depending only on $n$ and $d$.
In addition, the density $g(x)$ is chosen in such a way that for each $Q^\Gamma \in m(S^\Gamma)$,
\begin{align} \label{eqn:dens_just}
\mu(Q^\Gamma) = \HD^d|_E(Q^\Gamma) = \gamma(Q^\Gamma),
\end{align}
where we recall that we are assuming without loss of generality that $\mu = \HD^d|_E$.
With our definition of our approximating measure in hand, we move to the main estimate.

\medskip
\noindent \textbf{Step two}: We estimate the $\alpha_\mu$ by the $\alpha_\gamma$.
First, notice that if $x \in (E \Delta \Gamma) \cap Q^\Gamma(S)$, then $x$ belongs to some minimal cube of $S^\Gamma$ (since otherwise, being contained in arbitrarily small cubes of $S^\Gamma$, and (\ref{eqn:cor_alpha_st2p}) gives that $x \in E \cap \Gamma$).
From this, one deduces that for any cube $Q^\Gamma \in S^\Gamma$, 
\begin{align}
\left( (E \Delta \Gamma) \cap Q^\Gamma  \right) \subset \bigcup_{\substack{R^\Gamma \in m (S^\Gamma) \\ R^\Gamma \subset Q^\Gamma}} R^\Gamma. \label{eqn:cor_alpha_st2pp}
\end{align}

Now, let us begin our estimate.
Fix any cube $Q \in S$ for $E$.
Notice $ 3 B_Q \subset 10 B_{Q(S)} \subset Q^\Gamma(S)$, by choice of $Q^\Gamma(S)$.
Now since $\family$ is coherent, any $Q^\Gamma \in \Delta^\Gamma$ meeting $3 B_Q$ that also satisfies $M \diam Q^\Gamma \ge \diam Q$ must be so that $Q^\Gamma \in S^\Gamma$.
Indeed let us argue this by contradiction and suppose that this is not true.
Then there is some minimal cube $R^\Gamma \in S^\Gamma$ containing $Q^\Gamma$, and some sibling of $R^\Gamma$, $(R^\Gamma)' \in \Delta^\Gamma$ which is FS.
That is, $\lambda (R^\Gamma)'$ does not meet any $R \in S^*$ satisfying 
\begin{align*}
M^{-1} \diam (R^\Gamma)' \le \diam R \le M \diam (R^\Gamma)'.
\end{align*}
However, notice that as long as $\lambda $ is chosen large enough (depending on $M$), then the fact that $(R^\Gamma)'$ and $R^\Gamma$ are siblings, $R^\Gamma$ meets $3B_Q$, and $\diam Q \le M \diam Q^\Gamma \le M \diam (R^\Gamma)'$ implies that $\lambda (R^\Gamma)'$ meets $Q$.
Since $(R^\Gamma)'$ is FS, then $Q \in S$ and $\diam Q \le M \diam (R^\Gamma)'$ imply $\diam (R^\Gamma)' > M \diam Q$.
However, since $\family$ is coherent, and $\lambda (R^\Gamma)'$ meets $Q$, one may find some parent $Q' \supset Q$ with $Q' \in S$ so that $$M^{-1} \diam (R^\Gamma)' \le \diam Q' \le M \diam (R^\Gamma)',$$ which contradicts the fact that $(R^\Gamma)'$ is FS.
Hence, we have proved that for any cube $Q^\Gamma \in \Delta^\Gamma$,
\begin{align}
Q \in S, 3B_Q \cap Q^\Gamma \ne \emptyset, \text{ and } M \diam Q^\Gamma \ge \diam Q \Rightarrow Q^\Gamma \in S^\Gamma \label{eqn:cor_alpha_st3}.
\end{align} 
From here, we see that (\ref{eqn:cor_alpha_st3}) and the fact that $\Delta^\Gamma$ partitions the space $Q^\Gamma(S)$, that whenever $Q \in S$,
\begin{align}
3 B_Q \text{ meets some } Q^\Gamma \in S^\Gamma \text{ satisfying } \diam Q \le \diam Q^\Gamma \le 2 \diam Q. \label{eqn:cor_alpha_st4}
\end{align}
For every $Q \in S$, denote by any such choice of a cube $Q^\Gamma$ as in (\ref{eqn:cor_alpha_st4}) by $T(Q) \in S^\Gamma$.
Remark that since $\diam T(Q) \ge \diam Q$ for $Q \in S$, we know that $10 T(Q) \supset 3 B_Q$.

For each cube $Q \in S$, choose a flat measure $\nu_{Q}$ minimizing the quantity 
\begin{align*}
\tilde{\alpha }_\gamma (T(Q)) & \coloneqq (\diam T(Q))^{-d-1} \inf_{\nu \in \mathrm{Flat}(n,d) }   \mathcal{D}_{c_{T(Q)}, 10 \diam T(Q)}(\gamma, \nu),
\end{align*}
where $\mathrm{Flat}(n,d)$ is the space of flat measures as in Definition \ref{defn:alpha} and $\mathcal{D}_{x,r}$ is as in Definition \ref{defn:wasserstein}.
Then whenever $f \in \Lambda( 3B_Q)$,
\begin{align}
(\diam Q)^{-d-1} \abs{ \int f (d\mu - d\nu_{Q}) } & \le (\diam Q)^{-d-1} \left( \abs{ \int f (d \mu - d \gamma) } + \abs{ \int f (d \gamma - d \nu_Q) }   \right ) \nonumber \\ 
& \le (\diam Q)^{-d-1} \abs{ \int f (d \mu - d\gamma) } + C \tilde{\alpha }_\gamma (T(Q)). \label{eqn:cor_alpha_st5}
\end{align}
To estimate the first term above, we notice that by (\ref{eqn:cor_alpha_st2pp}), we have 
\begin{align*}
\abs{ \int f (d \mu - d \gamma) } & \le \sum_{ \substack{Q^\Gamma \in m(S^\Gamma) \\ Q^\Gamma \cap 3B_Q \ne \emptyset}  } \abs{ \int_{Q^\Gamma} f ( d\mu - d\gamma)  } \\
& \le \sum_{ \substack{Q^\Gamma \in m(S^\Gamma) \\ Q^\Gamma \cap 3B_Q \ne \emptyset}  } \abs{ \int_{Q^\Gamma} f - f(c_{Q^\Gamma}) d \mu  } + \abs{ \int_{Q^\Gamma} f - f(c_{Q^\Gamma}) d \gamma  } \\ 
& \le C \sum_{ \substack{Q^\Gamma \in m(S^\Gamma) \\ Q^\Gamma \cap 3B_Q \ne \emptyset}  } (\diam Q^\Gamma) \mu(Q^\Gamma), 
\end{align*}
since $\mu(Q^\Gamma) = \gamma(Q^\Gamma)$ for $Q^\Gamma \in m(S^\Gamma)$ (recall \eqref{eqn:dens_just}), and since $f$ is $1$-Lipschitz.
Combining the above with (\ref{eqn:cor_alpha_st5}) then gives
\begin{align*}
\sum_{Q \in S} & \alpha_\mu(Q)^2 \mu(Q) \\
&  \le C \sum_{ Q \in S }  \tilde{\alpha}_\gamma (T(Q))^2 \mu(Q)   + C \sum_{Q \in S} \mu(Q) (\diam Q)^{-2d-2} \Big( \sum_{ \substack{Q^\Gamma \in m(S^\Gamma), \\ Q^\Gamma \cap 3B_Q \ne \emptyset}  } (\diam Q^\Gamma) \mu(Q^\Gamma)  \Big)^2 \\
& \eqqcolon (\text{I}) + (\text{II}).
\end{align*}
To estimate (I), we remark that the diameter estimate on $T(Q)$, (\ref{eqn:cor_alpha_st4}), gives us that for some dimensional constant $C >0$, and uniformly over all $R^\Gamma \in S^\Gamma$, 
\begin{align}
\# \left \{ Q \in S \; : \; T(Q) = R^\Gamma \right \} \le C. \label{eqn:cor_alpha_st6}
\end{align}
Hence we readily see,
\begin{align*}
\sum_{Q \in S} \tilde{\alpha}_\gamma(T(Q))^2 \mu(Q)  & \le C \sum_{R^\Gamma \in \Delta^\Gamma(Q^\Gamma(S)) }  \tilde{\alpha}_\gamma(R^\Gamma)^2 \gamma(R^\Gamma)  \\
& \le C (M \delta)^{\theta_1} \gamma(Q^\Gamma(S))  \le C (M \delta)^{\theta_1} \mu(Q(S)).
\end{align*}
To be clear, the second to last inequality above follows from the estimate on the $d$-Ahlfors regularity constant of $\gamma$, (\ref{eqn:cor_alpha_4}), and the fact that $\Gamma$ is a $\delta$-Lipschitz graph.
Along with the fact that the cubes $Q^\Gamma \cap \Gamma$ for $Q^\Gamma \in \Delta^\Gamma$ serve as a system of dyadic cubes for $\Gamma$, the estimate then follows from Corollary (\ref{cor:alpha_lip}).
This gives our desired estimate on (I). 

We now move onto (II).
Define the collection of cubes $\mathcal{I}(Q) \subset m(S^\Gamma)$ for $Q \in S$ by $$ \mathcal{I}(Q) \coloneqq  \{ Q^\Gamma \in m(S^\Gamma) \; : \; Q^\Gamma \cap 3 B_Q \ne \emptyset, Q^\Gamma \subset R^\Gamma \in S^\Gamma \text{ for some } R^\Gamma \text{ satisfying (\ref{eqn:cor_alpha_st4})}  \}.$$
Notice that the same argument that precedes (\ref{eqn:cor_alpha_st3}) can be used to demonstrate that if $Q^\Gamma \in m(S^\Gamma)$ and $Q^\Gamma \cap 3B_Q \ne \emptyset$ for $Q \in S$, then necessarily, $\diam Q^\Gamma \le 2 \diam Q$.
Indeed, otherwise, the fact that $Q^\Gamma$ is minimal and choosing $\lambda$ large enough (depending on $M$) guarantees the stronger inequality $\diam Q^\Gamma > M \diam Q$.
But then the fact that $S$ is coherent can be used to show that none of the siblings of $Q^\Gamma$ are far from $S$, so $Q^\Gamma$ is not minimal.
Thus, $$Q^\Gamma \in m(S^\Gamma) \text{ and } Q^\Gamma \cap 3B_Q \ne \emptyset \Rightarrow \diam Q^\Gamma \le 2 \diam Q,$$ from which we can easily see we have equality of sets 
\begin{align}
\mathcal{I}(Q) = \{ Q^\Gamma \in m(S^\Gamma)\; : \; Q^\Gamma \cap 3 B_Q \ne \emptyset \}. \label{eqn:cor_alpha_8}
\end{align}
We now move on to the main estimate for (II).

We have the following string of inequalities:
\begin{align*}
(\text{II}) & = \sum_{Q \in S} \mu(Q) (\diam Q)^{-2d} \left( \sum_{ \substack{Q^\Gamma \in m(S^\Gamma) \\ Q^\Gamma \cap 3B_Q \ne \emptyset}  }  \left( \dfrac{\diam Q^\Gamma}{\diam Q} \right) \mu(Q^\Gamma)  \right)^2 \\
& \le C \sum_{Q \in S} (\diam Q)^{-d} \left( \sum_{ Q^\Gamma \in \mathcal{I}(Q) }  \left( \dfrac{\diam Q^\Gamma}{\diam Q} \right) \mu(Q^\Gamma)  \right)^2 \\
& \le C \sum_{Q \in S} \left( \sum_{Q^\Gamma \in \mathcal{I}(Q)} \mu(Q^\Gamma)  \left( \dfrac{\diam Q^\Gamma }{ \diam T(Q)} \right)^2  \right)  (\diam Q)^{-d} \left( \sum_{ Q^\Gamma \in \mathcal{I}(Q) } \mu(Q^\Gamma) \right) \\
& \le C \sum_{Q \in S} \left( \sum_{Q^\Gamma \in \mathcal{I}(Q)} \mu(Q^\Gamma)  \left( \dfrac{\diam Q^\Gamma }{ \diam T(Q)} \right)^2  \right) \\
& \le C \sum_{R^\Gamma \in S^\Gamma } \sum_{ \substack{ Q^\Gamma \in m(S^\Gamma) \, : \, Q^\Gamma \subset R^\Gamma \\ Q^\Gamma \text{ meets some $3B_Q$ where } Q \in S  } }  \mu(Q^\Gamma)  \left( \dfrac{\diam Q^\Gamma }{ \diam R^\Gamma} \right)^2    \\
& \le C \sum_{ \substack{ Q^\Gamma \in m(S^\Gamma) } }\mu(Q^\Gamma) \sum_{ \substack{R^\Gamma \in S^\Gamma \\ R^\Gamma \supset Q^\Gamma } } \left( \dfrac{\diam Q^\Gamma }{\diam R^\Gamma}  \right)^2    \le C \sum_{ \substack{ Q^\Gamma \in m(S^\Gamma)}  } \mu(Q^\Gamma).
\end{align*}
In the above, the first inequality holds simply because we sum the same cubes, by (\ref{eqn:cor_alpha_8}), and since $\mu(Q) \simeq_{n,d} (\diam Q)^d$.
The second holds by Cauchy-Schwarz and since $\diam Q \simeq_{n,d} \diam T(Q)$, and the third because the cubes in $\mathcal{I}(Q)$ are disjoint minimal cubes of $S^\Gamma$ contained in $10 B_Q$.
The fourth follows from (\ref{eqn:cor_alpha_st4}) and (\ref{eqn:cor_alpha_st6}), so that to each $Q$ there corresponds an $R^\Gamma \in S^\Gamma$ for which the term appears, and each such $R^\Gamma$ corresponds to only finitely many such $Q \in S$.
The fifth is just a switching of the order of summation, and the sixth is because the remaining inner series is a geometric series. 

In view of our estimate on (I), our last step of the proof is to verify that for some $\theta >0$ (depending only on $n$ and $d$),
\begin{align}
\sum_{ Q^\Gamma \in m(S^\Gamma) } \mu(Q^\Gamma) & \le C \delta^\theta \mu(Q(S)). \label{eqn:cor_alpha_st7}
\end{align}
In fact, we shall see that we may take $\theta = 1$. 

 \medskip
\noindent \textbf{Step three}: we prove (\ref{eqn:cor_alpha_st7}).
Let $Q^\Gamma \in m(S^\Gamma)$.
By definition, there is some $R^\Gamma \in S^\Gamma$, a sibling of $Q^\Gamma$ that is FS.
Recall that $R^\Gamma \subset Q^\Gamma(S) \subset 20 \sqrt{n} B_{Q(S)}$, and thus (\ref{eqn:cor_alpha_3}) implies that $R^\Gamma$ must meet one of the $R_0, R_1, \dotsc, R_k$, say $R_j$, as long as $\delta_0$ is small. Indeed \eqref{eqn:cor_alpha_3} says that such an $R^\Gamma$ must meet some portion of $E$ for $\delta_0$ small, and since the $R_i$ cover all of $E$ in a neighborhood of $R_0$, then $R^\Gamma$ must meet one of them. For convenience write $S_0 = S = S(R_0)$.
Recall that
\begin{align*}
\diam R^\Gamma & \le \diam Q^\Gamma(S) \le C \diam Q(S) \le C_2 \diam R_j, 
\end{align*}
for some dimensional constant $C_2 > 1$.
By taking $M \ge \max \{ C_2, C_D\}$, we see that it must be the case that $\diam R_j > M \diam R^\Gamma$, since otherwise $R^\Gamma$ is not FS.
Now we repeat this argument on the children of $R_j$.
Because $R^\Gamma$ meets $R_j$, we see that there is some $R_{j}^1 \subset R_j$ a child of $R_j$ that meets $R^\Gamma$.
Since $\diam R_j^1  \ge \cd^{-2} \diam R_j$, we know that $\diam R_j^1 \ge M C_D^{-2} \diam R^\Gamma \ge M^{-1} \diam R^\Gamma$ by the choice of $M$ above. Since $R^\Gamma$ is FS, either $R_j^1 \not \in S_j$, or, $R_j^1 \in S_j$ and in fact the stronger inequality $\diam R_j^1 > M \diam R^\Gamma$ holds.
We continue this argument finitely many times until we reach a child $R_j^\ell \subset R_j, R_j^\ell \in m(S_j)$, which meets $R^\Gamma$ and satisfies $\diam R_j^\ell > M \diam R^\Gamma$.
In particular, since $Q^\Gamma$ is a sibling of $R^\Gamma$, and since $\diam R^\Gamma < M^{-1} \diam R_j^\ell$, we readily see that $Q^\Gamma \subset 3 B_{R_j^\ell}$.
This shows that 
\begin{align}
\bigcup_{Q^\Gamma \in m(S^\Gamma)} Q^\Gamma \subset \bigcup_{j=0}^k \bigcup_{ Q \in m(S_j)} 3B_Q. \label{eqn:cor_alpha_st8}
\end{align}

Using (\ref{eqn:cor_alpha_st8}) and the fact that the $Q^\Gamma \in m(S^\Gamma)$ are disjoint, we then estimate
\begin{align*}
\sum_{Q^\Gamma \in m(S^\Gamma)} \mu(Q^\Gamma)  & = \mu \left( \bigcup_{Q^\Gamma \in m(S^\Gamma)} Q^\Gamma  \right)  \le  \mu \left( \bigcup_{j=0}^k \bigcup_{Q \in m(S_j)} 3 B_Q \right) \\
& \le \sum_{j=0}^k \sum_{Q \in m(S_j)} \mu(3 B_Q)   \le C \sum_{j=0}^k \sum_{Q \in m(S_j)} \mu(Q) \\
&  \le C \sum_{j=0}^k \sum_{ S \in \family(R_j), S \ne S_j  } \mu(Q(S))   \stackrel{(\ref{cond:corona_carleson_cor})}{\le} C \sum_{j=0}^k \delta \mu(R_j)  \le C (k+1) \delta \mu(R_0).
\end{align*}
Recalling that $k \le C$ for some dimensional constant $C$ depending only on $n$ and $d$ this shows (\ref{eqn:cor_alpha_st7}), and thus the proof of the Theorem is complete.
\end{proof}

\section{Reifenberg flatness implies \texorpdfstring{$\gamma(E)$}{oscillation of tangents} is small} \label{sec:reif_gamma}
In this section show that \ref{cond:reif} gives \ref{cond:bmo}, i.e., we prove Theorem \ref{thm:reif_bmo}.
First let us introduce in what sense $E$ will have tangent planes, which requires some notation. 
If $\mu_n, \mu$ are Radon measures on $\R^n$, then we write $\mu_n \wra \mu$ (and say, $\mu_n$ converges weakly to $\mu$) to mean that for each $\phi \in C_c(\R^n)$, $\int \phi \; d\mu_n \ra \int \phi \; d \mu$ as $n \ra \infty$.
Whenever $\mu$ is a Radon measure and $f: \R^n \ra \R^n$ is proper continuous map, then the Radon measure $\pushm{f}{\mu}$ is defined by $\pushm{f}{\mu}(A) \equiv \mu(f^{-1}(A))$.
Finally, denote for $x \in \R^n$ and $r >0$, $\Phi_{x,r}(y) = (y-x)/r$. 

\begin{thm}[see Theorem 10.2, \cite{MAGGI}] \label{thm:maggi_blowup}
Let $E \subset \R^n$ be $d$-rectifiable.
Then for $\HD^d$-almost all $x \in E$, there is a unique $d$-plane $T(x) \in G(n,d)$ so that 
\begin{align*}
r^{-d} \pushm{\left( \Phi_{x,r} \right) } {\HD^d|_E} \wra \HD^d|_{T(x)}.
\end{align*}
We call $T(x)$ the approximate tangent $d$-plane to $E$ at $x$.
\end{thm}

Finally, we introduce the quantity $\gamma(E)$ as in \cite{BLATT} which is the one appearing in Condition \ref{cond:bmo} in Theorem \ref{thm:main}.
Suppose that $E \subset \R^n$ is $d$-Ahlfors regular and $d$-rectifiable, so that $E$ has approximate $d$-planes $T(x)$ for $\HD^d$-almost all $x \in E$ by Theorem \ref{thm:maggi_blowup}.
Define
\begin{align*}
\gamma(E)  & \coloneqq \sup_{x \in E, r >0} \inf_{V \in G(n,d)}   \left(   \fint_{B(x,r)} \norm{ \pi_{T(x)} - \pi_V } \; d\HD^d|_E(x)  + \sup_{y \in B(x,r) \cap E} \dfrac{ \abs{\pi_{V^\perp}(y-x)}}{r} \right).
\end{align*}
Our intent in this section is to estimate this $\gamma(E)$ when $E$ is Reifenberg flat as in the following Theorem. In the next section, we will then use arguments from \cite{BLATT} to construct Lipschitz graph approximations to $E$ when $\gamma(E)$ is small.
\begin{thm}\label{thm:reif_bmo}
There are constants $C_0, \delta_0 >0$ depending only on $n$ and $d$ so that whenever $\delta \in (0, \delta_0)$, $E \subset \R^n$ is $d$-Ahlfors upper regular with constant $(1 + \delta)$ and $\delta$-Reifenberg flat, then $E$ is $d$-rectifiable, and moreover, $\gamma(E) \le C_0 \delta^{1/2}$. In addition, $E$ is lower $d$-Ahlfors regular with constant $(1+ C_0 \delta)$.
\end{thm}
\begin{proof}
Suppose that the hypotheses of the Theorem hold.
As long as $\delta_0 > 0$ is chosen sufficiently small (depending only on $n$ and $d$), then Theorem 15.2 in \cite{DTREIF} implies that $E$ satisfies the ``big pieces of Lipschitz graphs'' property, and thus is $d$-uniformly rectifiable (see the main Theorem in \cite{DSSIO}, or Theorem 1.57 in \cite{DSUR}).
Here we are using the fact that if $\delta_0$ is small enough, then then the Reifenberg flatness of $E$ guarantees that $E$ is lower $d$-Ahlfors regular with some bounded constant (this is made more precise in the following paragraph).
In particular, $E$ is $d$-rectifiable, and thus $T(x)$ exists for $\HD^d$-almost all $x \in E$.
Denote the set of all such $x$ by $E'$.

Fix $x \in E, r >0$, and denote by $P \in A(n,d)$ some choice of a $d$-plane so that $d_{x, 2r}(E, P) = b\beta_{\infty, E}( x, 2r) \le \delta$.
Notice that by translation invariance of the hypotheses and conclusion of the Theorem, we may as well assume that $P \in G(n,d)$.
Recall that $d_{x,r}$ is the normalized local Hausdorff distance, so that by definition,
\begin{align*}
\sup_{y \in E \cap \overline{B(x,2r)}} \dist(y, P) + \sup_{y \in P \cap \overline{B(x,2r)}} \dist(y, E) \le 2 \delta r.
\end{align*} 
Choose some point $p \in P$ so that $\abs{x-p} \le 2 \delta r$, and thus we have that $B(p, (1-2\delta)r) \subset B(x, r)$.
By Reifenberg's topological disk Theorem, we have that if $\delta_0$ is chosen sufficiently small depending only on $n$ and $d$, then necessarily $E$ is a $C^\alpha$-topological, $d$-disk \cite{REIF} (see also Section 3 in \cite{DKTDOUBLE} or \cite{REIFNEW} for other proofs).
In particular, since $E$ is very well approximated by $P$ in $B(x,2r)$, one can argue by contradiction to show that for $y \in P \cap B(p, r)$, there is some $x_y \in E \cap B(p, (3/2)r)$ so that $\pi_{P}(x_y) = y$, provided that $\delta_0$ is sufficiently small.
The argument is essentially contained in Lemma 8.3 of \cite{DKTDOUBLE}, so we omit the details.
This conclusion on projections also gives that $E$ is lower $d$-Ahlfors regular as follows.

Notice that in fact, the above implies that to each $y \in P \cap B(p, (1-4\delta)r)$, there is some $x_y \in E \cap B(p, (1-2\delta)r)$ for which $\pi_P(x_y) = y$.
Indeed, just choose the $x_y$ from above, so that $x_y \in E \cap B(p, (3/2)r)$.
Then we estimate 
\begin{align*}
\abs{x_y -p} & \le \abs{x_y - y} + \abs{y - p} \\
& = \abs{x_y - \pi_P(x_y)} + \abs{y - p} \\
& = \dist(x_y, P) + \abs{y-p} \\
& < 2\delta r +  ( 1- 4  \delta) r \\
& = ( 1- 2 \delta)r,
\end{align*}
so necessarily $x_y \in B(p, (1-2\delta)r)$. In particular, since $\pi_P$ is $1$-Lipschitz, we obtain the lower Ahlfors regularity of $E$:
\begin{align*}
\HD^d(E \cap B(x,r)) & \ge \HD^d(\pi|_P (E \cap B(x,r)) \\
& \ge \HD^d|_P(B(p, (1- 4 \delta)r)) \\
& = (1 - 4 \delta)^d r^d \\
& \ge (1 + C_0 \delta)^{-1} r^d,
\end{align*}
as long as $\delta_0$ is small enough and $C_0> 1$ is large enough. 

We are now ready to begin our estimate on the $\pi_{T(x)}$.
Fix $\epsilon > \delta$ to be determined, and set \[F \coloneqq \{x \in E' \cap B(p, (1-2\delta)r) \; : \; \norm{\pi_{T(x)} - \pi_P} > \epsilon \}, \, \tilde{F} \coloneqq \pi_P(F).\]
Define $\nu \coloneqq \pushm{(\pi_P)}{\HD^d|_{E}}$ to be the push-forward measure of $\HD^d|_E$ through the map $\pi_P$. Notice that a naive application of Chebysev gives us the estimate $\nu(F) \le C \epsilon^{-1} r^d$. The key step in estimating $\gamma(E)$ is to show the stronger estimate $\nu(F) \le C \delta \epsilon^{-1} r^d$ given that $E$ is $\delta$-Reifenberg flat. We claim that
\begin{align}
\liminf_{s \da 0} \dfrac{\nu(B(y,s))}{\HD^d|_P(B(y,s))} \ge (1 - \epsilon)^{-1} \label{eqn:reif_bmo_1}
\end{align} 
for each $y \in \tilde{F}$.
Let us show this now.

Let $y \in \tilde{F}$ so that there is some $x_y \in E' \cap B(p, (1- 2 \delta)r)$ with $\pi_P(x_y) = y$ and $\norm{\pi_{T(x_y)} - \pi_P } > \epsilon$.
Choose a unit vector $e_y \in \R^{n}$ for which $e_y \in T(x_y)$ but $\abs{\pi_P(e_y)} < 1 - \epsilon$.
We assume as well that $\pi_P(e_y) \ne 0$, but the argument when $\pi_P(e_y) = 0$ is similar, and in fact, one can show in this case that the left-hand side of (\ref{eqn:reif_bmo_1}) is $+ \infty$.
Choose unit vectors $e_1, \dotsc, e_{d-1} \in \R^n$ so that $\{ \pi_P(e_y)/\abs{\pi_P(e_y)}, e_1, \dotsc, e_{d-1} \}$ is an orthonormal basis for $P \subset \R^n$.
If we consider the set \[H_s(x_y) \coloneqq \left \{ x_y + \left(\beta_0 e_y + \sum_{i=1}^{d-1} \beta_i e_i \right) + v \; : \; v \in P^\perp, (1-\epsilon)^2 \beta_0^2 + \sum_{i=1}^{d-1} \beta_i^2 < s^2 \right\},\]
for $s \ll r$, then a direct computation of $\abs{y - \pi_P(z)} = \abs{\pi_P(x_y - z)}$ for $z \in H_s(x_y)$ using $\abs{\pi_P(e_y)}< 1- \epsilon$ gives that $\pi_P(H_s(x_y)) \subset P \cap B(y,s)$.
In particular, we obtain
\begin{align*}
\nu(B(y,s)) & = \HD^d|_E(\pi_P^{-1}(B(y,s) \cap P)) \nonumber \\
& \ge \HD^d|_E(H_s(x_y)),
\end{align*}
so that 
\begin{align}
\dfrac{\nu(B(y,s))}{\HD^d|_P(B(y,s))} & = \dfrac{\nu(B(y,s))}{s^d} \nonumber \\
& \ge \dfrac{\HD^d|_E(H_s(x_y))}{s^d} \nonumber  \\
& = \dfrac{\HD^d|_E( s H + x_y ) }{s^d}, \label{eqn:reif_bmo_2}
\end{align}
where $s H + x_y \equiv H_s(x_y)$ defines the set $H$.
However, $H$ contains a $d$-ellipsoid $A \subset T(x_y)$ of the form 
\begin{align*}
A \coloneqq \{ \beta_0 e_y + \sum_{i=1}^{d-1} \beta_i e_i' \; : \; (1 - \epsilon)^2 \beta_0^2 + \sum_{i=1}^{d-1} \beta_i^2 < 1 \},
\end{align*}
where $\{e_y, e_1', \dotsc, e_{d-1}'\}$ is some orthonormal basis for $T(x_y)$.
Thus, in view of Theorem \ref{thm:maggi_blowup}, (\ref{eqn:reif_bmo_2}), and the fact that $\HD^d|_{T(x_y)}(\partial A) = 0$, we obtain that 
\begin{align}
\liminf_{s \da 0} \dfrac{\nu(B(y,s))}{\HD^d|_P(B(y,s))} & \ge \lim_{s \da 0} \dfrac{\HD^d|_E( s H + x_y ) }{s^d} \nonumber \\
& = \lim_{s \da 0} s^{-d}\pushm{(\Phi_{x_y, s})}{\HD^d|_E} \left( H \right)  \nonumber \\
&  = \HD^d|_{T(x_y)}(A) \nonumber \\
& = (1 - \epsilon)^{-1}, \label{eqn:rad_low}
\end{align}
proving (\ref{eqn:reif_bmo_1}).
Here we have used the fact that if $\mu_n \wra \mu$ and $B \subset \R^n$ is Borel and bounded with $\mu(\partial B) = 0$, then $\mu_n(B) \ra \mu(B)$ (see for example, Proposition 4.26 in \cite{MAGGI}). We are also using our convention that $\HD^d$ is normalized so that if $B$ is a ball of radius $r$ centered on $T(x_y)$, then $\HD^d|_{T(x_y)}(B) = r^d$, which gives the value of $\HD^d|_{T(x_y)}(A)$ above.

Finally, we get an estimate on the size of $\HD^d|_E(F)$.
First, Theorem 2.13(2) in \cite{Mattila} along with estimate \eqref{eqn:rad_low} implies
\begin{align}
\HD^d|_E(F)  = \nu ( \tilde{F})  & \ge \int_{\tilde{F}} \limsup_{r \da 0} \dfrac{\nu(B(y,r))}{\HD^d|_P(B(y,r))} \; d\HD^d|_P(y)  \ge (1-\epsilon)^{-1} \HD^d|_P(\tilde{F}), \label{eqn:reif_bmo_3}
\end{align}
(in particular, we do not need $\nu \ll \HD^d|_P$ to get the first inequality above).
Denote $G = E \cap B(p, (1-2r)) \setminus F$.
Recalll that $\pi_P(E \cap B(p, (1- 2 \delta)r)) \supset P \cap B(p, (1-4\delta)r)$ and the fact that $\HD^d|_E( B(p, (1-2\delta)r) \setminus (F \cup G)) = 0$. Hence  the inequality (\ref{eqn:reif_bmo_3}), the fact that $\HD^d|_E$ is upper $d$-Ahlfors regular with constant $(1+\delta)$, and the fact that $\pi_P$ is $1$-Lipschitz gives
\begin{align*}
(1 - 4\delta)^d r^d & = \HD^d|_P(B(p, (1-4\delta)r))) \\
 & \le \HD^d(\pi_P(G)) + \HD^d|_P(\tilde{F}) \\
 & \le \HD^d|_E(G) + (1- \epsilon) \HD^d|_E(F) \\
 & = \HD^d|_E(B(p, (1-2\delta)r)) - \epsilon \HD^d|_E(F) \\
 & \le (1 + \delta) (1-2\delta)^d r^d - \epsilon \HD^d|_E(F). 
\end{align*}
Rearranging for $\HD^d|_E(F)$ yields that 
\begin{align*}
\HD^d|_E(F) & \le \epsilon^{-1} \left( (1 + \delta) - (1 - 4\delta)^d) \right)r^d  \le C \delta \epsilon^{-1} r^d,
\end{align*}
where $C$ is some constant depending only on $d$.
Thus since $B(x, (1- 4\delta)r) \subset B(p, (1- 2 \delta)r)$, 
\begin{align*}
\int_{B(x, (1-4\delta)r)} \norm{\pi_{T(x)} - \pi_P} \; d\HD^d|_E & \le \int_{G} \norm{ \pi_{T(x)} - \pi_P} \; d\HD^d|_E + \int_{F} \norm{ \pi_{T(x)} - \pi_P } \; d\HD^d|_E \\
& \le \epsilon \HD^d|_E(B(x, r)) + 2 \HD^d(F) \\
& \le C (\epsilon + \delta \epsilon^{-1}) r^d  \\
& \le C \delta^{1/2} r^d
\end{align*}
by taking $\epsilon = \delta^{1/2}$. 

Now we take on the second term of $\gamma(E)$; this estimate comes directly by choice of $P$.
Notice that for any $y \in B(x, (1-4\delta)r) \cap E$ we have that 
\begin{align*}
\abs{ \pi_{P^\perp}(y-x)} & = \abs{y - \pi_P(y) - x + \pi_P(x) } \\
& \le \abs{ y - \pi_P(y)} + \abs{ x - \pi_P(x) } \\
& \le 2 \delta r,
\end{align*}
by choice of $P$.
Thus, altogether we've shown that 
\begin{align*}
\int_{B(x, (1-4 \delta)r)} \norm{\pi_{T(x)} - \pi_P } \; d\HD^d|_E + \sup_{y \in E \cap B(x, (1-4\delta)r)} \dfrac{\abs{\pi_{P^\perp}(y-x)}}{r} \le C \left(\delta^{1/2} + C \delta\right) r^d,
\end{align*}
from which we readily see that $\gamma(E) \le C \delta^{1/2}$. \qedhere
\end{proof}
\section{\texorpdfstring{$\gamma(E)$}{oscillation of tangents} small implies good Lipschitz approximations to \texorpdfstring{$E$}{E}} \label{sec:gamma_ur}

In this last section, we briefly detail how the work of \cite{BLATT} demonstrates that \ref{cond:bmo} gives \ref{cond:delta_ur}.
The argument involving the Hardy-Littlewood Maximal function goes back to the co-dimension 1 case in \cite{CASI}.
Under different assumptions on a domain $\Omega$ when $\partial \Omega$ is $(n-1)$-Ahlfors regular, there is also a proof in \cite{HMT10}.
Our goal here is to point the reader to the fact that for these particular arguments in \cite{BLATT}, one does not need $E$ to be a $C^1$ $d$-dimensional chord-arc submanifold, as long as we instead assume Ahlfors regularity of $E$.
Since all of the arguments exist in this work, we only enumerate the various steps in the proof.
Let us state precisely the Theorem that can be obtained.

\begin{thm}[Theorem 3.1 in \cite{BLATT}] \label{thm:gamma_ur}
There are constants $\delta_0, C_0 >0$ depending only on $n$, $d$, and $C_E$, so that whenever $E$ is lower $d$-Ahlfors regular with constant $(1+\delta)$, upper $d$-Ahlfors regular with constant $C_E >0$, and $\gamma(E) \le \delta < \delta_0$, then $E$ is $(C_0 \delta^{1/2})$-UR. 
\end{thm}
\begin{proof}
Recall that $\gamma(\cdot)$ is only defined on $d$-rectifiable subsets, so we assume $\gamma(E) \le \delta$ and $E$ is $d$-rectifiable.
Thus Theorem \ref{thm:maggi_blowup} applies so that $E$ has approximate tangent $d$-planes almost everywhere in $E$.
We continue as in the proof of Lemma 3.2 of \cite{BLATT}. 

Fix $x_0 \in E, R >0$, and $\tau \in (10 \delta, 1/3)$.
Denote by $\mathcal{M}_R$ for $R >0$ the variant of the Hardy-Littlewood Maximal function, $\mathcal{M}_Rf(x) = \sup_{0 < r < R} \fint_{B(y,r)} \abs{f} \; d\HD^d|_E$ for $x \in E$ and $\mathcal{M}_R f(x) =0 $ otherwise. 

Set
\begin{align*}
F & \coloneqq \{y \in B(x_0,R) \cap E \; : \; \mathcal{M}_{4R}( \norm{ \pi_T - \pi_{T_{x_0, 4R}}} )(y) \le \tau \}, \\
B & \coloneqq (B(x_0, R) \cap E) \setminus F,
\end{align*}
where $T(y)$ is the approximate tangent $d$-plane to $E$ at $y$ and $T_{x_0, 4R} \in G(n,d)$ is a $d$-plane minimizing the quantity 
\begin{align*}
 \fint_{B(x_0, 4R)}  \norm{ \pi_{T(y)} - \pi_V } \; d\HD^d|_E(x)  + \sup_{y \in B(x_0,4R) \cap E} \dfrac{ \abs{\pi_{V^\perp}(y-x_0)}}{4R}
\end{align*}
over all $V \in G(n,d)$.
Then Step 1 in \cite[Lemma 3.2]{BLATT} (which only uses the Ahlfors regularity of $E$, and the definition of $\gamma(E)$) gives that for $\delta_0$ chosen small enough, there are uniform constants $a, C>0$ depending only on $n$, $d$, and $C_E$ so that 
\begin{align}
\HD^d(B) \le C e^{-a \tau/\delta} R^d.
\end{align}
Steps 2 and 4 of the proof of the same Lemma (which again, only use Ahlfors regularity of $E$ and Step 1) then give that for $\delta_0$ sufficiently small, there is a Lipschitz graph $\Gamma$ with constant $\le C \tau$ for which $F \subset \Gamma$.
Consequently, setting $\tau = \delta^{1/2}$ gives
\begin{align}
\HD^d((B(x_0, R) \cap (E \setminus \Gamma) ) & \le C e^{-a \delta^{-1/2}} R^d \nonumber \\
& \le C \delta R^d. \label{eqn:gamma_ur_1}
\end{align}
The fact that $\HD^d|_E$ is lower $d$-Ahlfors regular with constant $(1 + \delta)$ and the norm on the Lipschitz constant of $\Gamma$ readily give
\begin{align*}
\HD^d(B(x_0, R) \cap (\Gamma \setminus E)) & = \HD^d(\Gamma \cap B(x_0, R)) - \HD^d(E \cap B(x_0, R)) \\
& \le \left ( \sqrt{1 + C \delta} - ( 1 + \delta)^{-1} \right) R^d \\
& \le C \delta^{1/2} R^d.
\end{align*}
Together with \eqref{eqn:gamma_ur_1}, this shows that $E$ is $(C_0 \delta^{1/2})$-UR.
\end{proof}

In all, Theorem \ref{thm:gamma_ur} thus concludes the proof of Theorem \ref{thm:main}.
\section{A comment on local results and chord-arc domains with small constant} \label{sec:local}
As mentioned in the introduction, Theorem \ref{thm:main} has local and ``vanishing'' versions, which correlate more closely to the local definition of $\delta$-chord-arc domains as in \cite{KT99}. Instead of formulating very precise local definitions here, we simply remark the following about the proofs of Theorems \ref{thm:ur_implies_corona}, \ref{thm:corona_alpha}, \ref{thm:reif_bmo}, and \ref{thm:gamma_ur}.
In each of these proofs, the conclusion of the Theorem is deduced inside a ball $B(x,r)$ centered on $E$ using information about $E$ in the ball $B(x, C_0r)$ up to scale $C_0 r$ for some dimensional constant $C_0 = C_0(n,d, C_E) >0$, \textit{except} for Theorem \ref{thm:ur_implies_corona}.
In the argument of Theorem \ref{thm:ur_implies_corona}, we used information about $E$ inside the larger ball $B(x, C_0 r)$ up to the scale $ C_0 \delta^{-\theta'} r$ where $\theta' \in (0, 1/4d)$. 
This means that the local version of Theorem \ref{thm:main} should be loosely formulated in the following way. 

\begin{thm}[Local version of Theorem \ref{thm:main}]\label{thm:main_loc}
Fix $n, d \in \N$ with $0 < d < n$ and $C_E >0$. Then there are constants $\tau_0, \theta_0, \delta_0 \in (0,1)$ depending only on $n, d$ and $C_E$ so that the following holds.

Suppose that $E$ is a set which satisfies 
\begin{align*}
C_E^{-1} r^d \le \HD^d(E \cap B(x,r)) \le C_E r^d
\end{align*}
for each $x \in E \cap B(0,R)$ and $0 < r < r_0$. Assume in addition that any one of the conditions \ref{cond:delta_ur}-\ref{cond:bmo} hold for $x \in E \cap B(0, R)$ and $0 < r < r_0$ with constant $\delta \in (0, \delta_0)$. Then the rest hold for all points $x \in E \cap B(0, \tau_0 R_0)$ and scales $0 < r < \delta \tau_0 r_0$ with constant $\delta^{\theta_0}$.
\end{thm}

Here, the phrase ``condition \ref{cond:delta_cor} holds for $x \in E \cap B(0, R)$ and $0 < r < r_0$ with constant $\delta$'' really means that for each $Q_0 \in \Delta$ with $\dist(Q_0, B(0, R)) \le r_0$ and $\diam Q_0 \le r_0$, $E$ admits $\delta$-Corona decompositions in $Q_0$. For the others, the meaning is self-explanatory: we just mean the conditions defining the statement are required to hold only for such points $x \in E$ and such scales $r > 0$ as opposed to uniformly.

In particular, this remark can be used to prove Theorem \ref{thm:cad} in the following way. 

\begin{proof}[Proof of Theorem \ref{thm:cad}]
Let $\Omega \subset \R^n$ be such a domain as in the statement of the Theorem, and for convenience write $\sigma \coloneqq \HD^{n-1}|_{\partial \Omega}$. Fix $\theta_0$, $\tau_0$, and $\delta_0$ coming from Theorem \ref{thm:main_loc} depending on $n$ and $C_E$, and let $\delta < \delta_0$.

Suppose first that $\Omega$ is a $\delta$-chord arc domain. Fix a ball $B(0, R)$ with $R >1$ large, and assume without loss of generality that $0 \in \partial \Omega$. Then we may find some $\rho >0$ small so that for  $x \in \partial \Omega \cap B(0, R)$ and $r \in (0, \rho)$,
\begin{align*}
b\beta_{\infty, \partial \Omega}(x,r) ,  \; \norm{\vec{n}}_{*}(B(x, \rho))  \le \delta.
\end{align*}
This first local Reifenberg condition on $\partial \Omega$ gives, by the proof of Theorem \ref{thm:reif_bmo}, that 
\begin{align*}
\sigma(B(x,r)) \ge (1 + C \delta)^{-1} r^{n-1},
\end{align*}
whenever $x \in B(0, \tau_0 R)$ and $r \le \tau_0 \rho$ for some constant $C>0$ depending only on $n$. Moreover, \cite[equation (2.18)]{BEGTZ22} implies the estimate $\abs{\langle \vec{n}_{x,r}, y-x\rangle} \le C r \delta^{1/2}$ whenever $x \in \partial \Omega \cap B(0, R)$, $r \in (0, \rho)$ and $y \in \partial \Omega \cap B(x, r)$. Here $C$ is a constant depending only on $n$ and $C_E$. Combined with the lower Ahlfors regularity condition above, this says exactly that $\partial \Omega$ satisfies condition \ref{cond:bmo} for points $x \in \partial \Omega \cap B(0, \tau_0 R)$ and scales $r \in (0, \tau_0 \rho)$ with constant $C \delta^{1/2}$. Theorem \ref{thm:main_loc} then implies that \ref{cond:alpha_c} holds for points $x \in \partial \Omega \cap B(0, \tau_0^2 R)$ and scales $r \in (0, \delta \tau_0^2 \rho)$ with constant $C \delta^{\theta_0/2}$. Since $R> 0$ is arbitrary, this shows that \ref{cond:cad} implies \ref{cond:loc_carl} with worse constant, $\delta^{\theta_0'}$ for some $\theta_0'$ small.

Conversely, assume that $\Omega$ satisfies \ref{cond:loc_carl}.
Again fix a ball $B(0, R)$ with $R >1 $ large, and assume without loss of generality that $0 \in \partial \Omega$.
By assumption, there is some $\rho >0$, so that the measure $\sigma$ satisfies the small constant Carleson measure condition
\begin{align*}
\sigma(B(x,r))^{-1} \int_{B(x,r)} \int_0^r \alpha_\sigma(y, s)^2 \; \dfrac{d\sigma(y)ds}{s} \le \delta,
\end{align*}
for all $x \in \partial \Omega \cap B(0, R)$ and $r \in (0, \rho)$.
Also, $\sigma(B(x,r)) \le (1+\delta)r^{n-1}$ for such $x$ and $r$.
In other words, $\partial \Omega$ satisfies condition \ref{cond:alpha_c} for $x \in B(0, R)$ and $r \in (0, \rho)$ with constant $\delta$, so Theorem \ref{thm:main_loc} implies that for $x \in \partial \Omega \cap B(0, \tau_0 R)$ and $r \in (0, \delta \tau_0 \rho)$, $\partial \Omega$ satisfies conditions \ref{cond:reif} and \ref{cond:bmo} with constant $\delta^{\theta_0}$. In other words, for each $x \in \partial \Omega \cap B(0, \tau_0 \rho)$ and $r \in (0, \delta \tau_0 \rho)$, we have 
\begin{align*}
b\beta_{\infty, \partial \Omega}(x,r) \le \delta^{\theta_0}, \text{  and  }  \norm{\vec{n} }_{*}B(x, \delta \tau_0 \rho) \le \delta^{\theta_0}.
\end{align*}
Since $R >0$ is arbitrary, then joint with the underlying assumptions on $\Omega$ made in the statement of the Theorem, this implies that $\Omega$ is a $\delta^{\theta_0}$-chord-arc domain. 
\end{proof}

\bibliographystyle{alpha}
\bibliography{bibl}

\end{document}